\documentclass[ngerman,english]{amsart}
\usepackage[T1]{fontenc}
\usepackage[latin9]{inputenc}
\usepackage{amsthm}
\usepackage{amstext}
\usepackage{graphicx}
\usepackage{amssymb}
\usepackage[numbers]{natbib}

\makeatletter
\numberwithin{equation}{section}
\numberwithin{figure}{section}
\theoremstyle{plain}
\newtheorem{thm}{Theorem}[section]
  \theoremstyle{definition}
  \newtheorem{defn}[thm]{Definition}
  \theoremstyle{remark}
  \newtheorem{rem}[thm]{Remark}
  \theoremstyle{plain}
  \newtheorem{prop}[thm]{Proposition}
  \theoremstyle{plain}
  \newtheorem{lem}[thm]{Lemma}
  \theoremstyle{plain}
  \newtheorem{cor}[thm]{Corollary}
 \theoremstyle{definition}
  \newtheorem{example}[thm]{Example}
  \theoremstyle{plain}
  \newtheorem{fact}[thm]{Fact}

\usepackage{bbm}

\newcommand{\1}{\mathbbm{1}}
\newcommand{\N}{\mathbb{N}}

\newcommand{\Z}{\mathbb{Z}}

\newcommand{\R}{\mathbb{R}}

\newcommand{\e}{\mathrm{e}}

\renewcommand{\Pi}{\pi}
\renewcommand{\tilde}{\widetilde}

\renewcommand{\lg}{\log}

\DeclareMathOperator*\interior{int}

\newcommand\id{\mathrm{id}}

\DeclareMathOperator*{\card}{card}

\makeatother

\usepackage{babel}

\begin{document}
\selectlanguage{english}

\title{Induced topological pressure for countable state Markov shifts }

\author{Johannes Jaerisch}

\email{jogy@math.uni-bremen.de}

\author{Marc Kesseböhmer }

\email{mhk@math.uni-bremen.de}

\author{Sanaz Lamei }

\email{lamei@guilan.ac.ir}

\thanks{SL was supported by a grant offered by the University of Guilan. }

\address{Fachbereich Mathematik, Universität Bremen, 28359 Bremen, Germany }

\address{Department of Mathematics, The University of Guilan, P.O. Box 1841,
Rasht, Iran.}
\begin{abstract}
We introduce the notion of induced topological pressure for countable
state Markov shifts with respect to a non-negative scaling function
and an arbitrary subset of finite words. Firstly, the scaling function
allows a direct access to important thermodynamical quantities, which
are usually given only implicitly by certain identities involving
the classically defined pressure. In this context we generalise Savchenko's
definition of entropy for special flows to a corresponding notion
of topological pressure and show that this new notion coincides with
the induced pressure for a large class of Hölder continuous height
functions not necessarily bounded away from zero. Secondly, the dependence
on the subset of words gives rise to interesting new results connecting
the Gurevi\v c and the classical pressure with exhausting principles
for a large class of Markov shifts. In this context we consider dynamical
group extentions to demonstrate that our new approach provides a useful
tool to characterise amenability of the underlying group structure. 
\end{abstract}

\keywords{thermodynamical formalism; topological pressure; countable state
Markov shift; special flows; group extensions }

\subjclass[2000]{37D35; 28A65}

\maketitle

\section{Introduction and statement of main results}

Throughout this paper we consider a \emph{topological Markov shift}
\[
\Sigma:=\left\{ \omega:=\left(\omega_{1},\omega_{2},\ldots\right)\in I^{\N}:\forall i\in\N\; A_{\omega_{i}\omega_{i+1}}=1\right\} \]
with \emph{alphabet} $I\subset\N$, \emph{incidence matrix} $A\in\left\{ 0,1\right\} ^{I\times I}$,
and \emph{left shift map} $\sigma:\Sigma\rightarrow\Sigma$ given
by $\left(\omega_{1},\omega_{2},\ldots\right)\mapsto\left(\omega_{2},\omega_{3},\ldots\right)$.
The shift map $\sigma$ is continuous with respect to the metric $d_{\alpha}\left(\omega_{1},\omega_{2}\right):=\e^{-\alpha\left|\omega_{1}\wedge\omega_{2}\right|}$,
$\alpha>0$, where $\omega_{1}\wedge\omega_{2}$ denotes the longest
common initial block of $\omega_{1}$, $\omega_{2}\in\Sigma$. We
denote the set of $A$--admissible words of length $n\in\mathbb{N}$
 by $\Sigma^{n}:=\left\{ \omega\in I^{n}:\,\,\forall i\in\left\{ 1,\dots,n-1\right\} :A_{\omega_{i}\omega_{i+1}}=1\right\} $
and the set of $A$--admissible words of arbitrary length by $\Sigma^{*}:=\bigcup_{n\in\N}\Sigma^{n}$
. For $\omega\in\Sigma^{*}$ we let $|\omega|$ denote the (word)
length of $\omega$, which is the unique $n\in\N$ such that $\omega\in\Sigma^{n}$,
and for $\omega\in\Sigma$ we set $\left|\omega\right|=\infty$. For
$\omega\in\Sigma^{n}$ we call $\left[\omega\right]:=\left\{ \tau\in\Sigma:\tau_{|n}=\omega\right\} $
the \emph{cylindrical set} of $\omega$. 

Let us now introduce the induced topological pressure, which is the
central object of this paper. 
\begin{defn}
For $\varphi,\psi:\Sigma\rightarrow\R$ with $\psi\ge0$, and $\mathcal{C}\subset\Sigma^{*}$
we define for $\eta>0$ the $\psi$\emph{--induced pressure of} $\varphi$
(with respect to $\mathcal{C}$) by \[
\mathcal{P}_{\psi}\left(\varphi,\mathcal{C}\right):=\limsup_{T\rightarrow\infty}\frac{1}{T}\log\sum_{{\omega\in\mathcal{C}\atop T-\eta<S_{\omega}\psi\le T}}\exp S_{\omega}\varphi,\]
which takes values in $\overline{\R}:=\R\cup\left\{ \pm\infty\right\} $.
In here, we set $S_{\omega}\varphi:=\sup_{\tau\in\left[\omega\right]}\sum_{k=0}^{|\omega|-1}\varphi\circ\sigma^{k}\left(\tau\right)$.
\end{defn}
It will follow from Theorem \ref{thm:pressure_as_inf} that the definition
of $\mathcal{P}_{\psi}\left(\varphi,\mathcal{C}\right)$ is in fact
independent of the choice of $\eta>0$. For this reason we are not
going to refer to $\eta>0$ in the definition of the induced pressure. 

Of particular interest for the choice of $\mathcal{C}$ is certainly
the set of all finite words $\Sigma^{*}$ as well as the important
subsets

\[
\Sigma^{\mathrm{per}}:=\left\{ \omega\in\Sigma^{*}:\overline{\omega}\in\Sigma\right\} ,\qquad\Sigma_{a}^{\mathrm{per}}:=\left\{ \omega\in\Sigma^{\mathrm{per}}:\omega_{1}=a\right\} .\]

In here, for $\omega=(\omega_{1},\ldots,\omega_{n})\in\Sigma^{*}$,
we let $\overline{\omega}:=(\omega_{1},\ldots,\omega_{n},\omega_{1},\ldots,\omega_{n},\ldots)$
denote the periodic word in $I^{\N}$ with period $n\in\N$ and initial
block $\omega$.

Our definition generalises known notions of topological pressure for
the particular choice $\psi=1$. That is, if $A$ is mixing (i.e.
$\forall i,j\in I\exists n_{0}\ge1\forall n\ge n_{0}\exists\omega\in\Sigma^{n}:i\omega j\in\Sigma^{*}$)
then $\mathcal{P}_{1}\left(\varphi,\Sigma_{a}^{\mathrm{per}}\right)$
coincides with the Gurevi\v c pressure (cf.\ \citep{MR1738951,MR1818392}).
Whereas $\mathcal{P}_{1}\left(\varphi,\Sigma^{*}\right)$ coincides
with the classical notion of pressure for countable state Markov systems
(cf.\ \citep{MR2003772}). It is important to note that in the context
of the classical pressure one typically imposes that $A$ is \emph{finitely
irreducible}, i.e. there exists a finite set $\Lambda\subset\Sigma^{*}$
such that $\forall i,j\in I\exists\omega\in\Lambda:i\omega j\in\Sigma^{*}$.
While in the context of the Gurevi\v c pressure one typically imposes
that $A$ is mixing and satisfies the \emph{big images and preimages
(BIP) property,} i.e. there exists a finite set $B\subset I$ such
that $\forall i\in I\exists b_{1},b_{2}\in B:b_{1}ib_{2}\in\Sigma^{*}$.
Our general results make only few assumptions on the mixing properties
of the Markov shift. We will show in Proposition \ref{fac:pressure}
that several properties known for the $1$--induced pressure allow
a natural generalisation to the induced pressure.

Alternative approach to implicitly defined thermodynamic quantities

One important feature of the induced pressure is the freedom to choose
a scaling function $\psi\ge0$. This provides us with a new access
to important thermodynamical quantities, which were so far only given
implicitly as pseudo-inverse of the $1$--induced pressure before.
More precisely, under certain conditions (cf.\ Corollary \ref{cor:pressure_rescaling},
\ref{cor:pressureformula_ifstreactlydescreasing} and \ref{cor:psipressure_asinf-underexhaustingprop})
we have \[
\mathcal{P}_{\psi}\left(\varphi,\mathcal{C}\right)=\inf\left\{ \beta:\mathcal{P}_{1}\left(\varphi-\beta\psi,\mathcal{C}\right)\le0\right\} .\]
In fact, the right-hand side is often referred to as the free energy,
which is the main ingredient to determine multifractal spectra \citep{JaerischKessebohmer:09},
large deviation asymptotics \citep{MR1819804}, or the thermodynamic
formalism for special flows as elaborated in Remark \ref{rem:BarreiraIommi}
below. Furthermore, for the Gurevi\v c case (i.e. $\mathcal{C}:=\Sigma_{a}^{\mathrm{per}}$)
we are able to provide in Proposition \ref{pro:variationalprinciple-1-1}
a variational principle, which will turn out to be crucial for our
study of special semi-flows. 

Special semi-flows over countable state Markov shifts

In this paragraph we will discuss the connection between the induced
pressure and a new notion of pressure generalising Savchenko's definition
of entropy. Let $\tau:\Sigma\rightarrow\R_{>0}$ be a continuous function
and consider \[
Y:=\left\{ \left(\omega,t\right)\in\Sigma\times\R^{+}:0\leq t\leq\tau(\omega)\right\} \slash\sim,\]
 where we identify points via the equivalence relation given by $\left(w,\tau(w)\right)\sim\left(\sigma(w),0\right)$.
Let $\Phi=\left(\varphi_{t}\right)_{t\in\R_{>0}}$ be the special
semi-flow over $Y$ with height function $\tau$ defined by $\varphi_{t}(\omega,s)=(\omega,s+t)$.

Savchenko was the first in \citep{MR1627271} to define the topological
entropy of special flows for non-compact spaces to be\[
h\left(\Phi\right):=\sup\left\{ h\left(\varphi_{1},\mu\right):\mu\in\mathcal{E}_{\Phi}^{1}\right\} ,\]
where\foreignlanguage{ngerman}{ $\mathcal{E}_{\Phi}^{1}$} denotes
the set of $\Phi$-invariant ergodic probability measures on $Y$.
Savchenko then proved in \citep{MR1627271} that, if $\tau$ depends
only on the first coordinate and, then we have \[
h\left(\Phi\right)=\inf\left\{ \beta:\sum_{\omega\in\Sigma_{a}^{\mathrm{smpl}}}e^{-\beta S_{|\omega|}\tau}\le1\right\} ,\]
where $\Sigma_{a}^{\mathrm{smpl}}:=\left\{ \omega\in\Sigma_{a}^{\mathrm{per}}:\forall k\in\N:1<k\leq|\omega|:\omega_{k}\neq a\right\} $. 

This fact has been employed several times in the context of subgroups
of the modular group; e.g.\ in \citep{MR1901076}, \citep{Iommi2010}.
In \citep{DastjerdiLamei,DastjerdiLameib} some techniques have been
developed for computing the entropy of certain flows explicitly.

As a natural generalisation we may define the topological pressure
of a continuous function $g:Y\rightarrow\R$ with respect to the special
flow to be \[
\mathbf{P}\left(g\mid\Phi\right):=\sup\left\{ h\left(\varphi_{1},\mu\right)+\int g\, d\mu:\mu\in\mathcal{E}_{\Phi}^{1}\:\mbox{such that }g\in L^{1}\left(\mu\right)\right\} .\]
This gives in particular $h\left(\Phi\right)=\mathbf{P}\left(0\mid\Phi\right)$.

With $\Delta_{g}:\Sigma\rightarrow\R$ given by $\Delta_{g}\left(\omega\right):=\int_{0}^{\tau\left(\omega\right)}g\left(\omega,t\right)d\lambda\left(t\right)$
we are in the position to state our main result concerning special
flows. 
\begin{thm}
\label{thm:taupressure_is_entropy}If $\tau:\Sigma\rightarrow\R_{>0}$
is Hölder continuous satisfying $\sum_{i=0}^{\infty}\tau\circ\sigma^{i}=\infty$
and $g:Y\rightarrow\R$ has the property that $\Delta_{g}:\Sigma\rightarrow\R$
is Hölder continuous then we have\[
\mathbf{P}\left(g\mid\Phi\right)=\sup_{a\in I}\mathcal{P}_{\tau}\left(\Delta_{g},\Sigma_{a}^{\mathrm{per}}\right)=\sup_{{Y\supset C\textrm{ compact,}\atop \Phi\textrm{-invariant}}}\mathbf{P}\left(g\big|_{C}\mid\Phi\big|_{C}\right).\]
In particular, if $A$ is irreducible (i.e. $\forall i,j\in I$ $\exists\omega\in\Sigma^{*}:$
$i\omega j\in\Sigma^{*}$) then $\mathbf{P}\left(g\mid\Phi\right)=\mathcal{P}_{\tau}\left(\Delta_{g},\Sigma_{a}^{\mathrm{per}}\right)$
for any $a\in I$. 
\end{thm}
The proof of Theorem \ref{thm:taupressure_is_entropy} will be postponed
to Section \ref{sec:Proof Special-semi-flows}. It relies on some
inducing techniques which allow us to overcome the difficulty of certain
marginal measures on the base space to be infinite (see Lemma \ref{lem:entropylemma}).
\begin{rem}
\label{rem:BarreiraIommi} Using Corollary \ref{cor:ExhaustGurevich}
we find that our definition of pressure for special flows in fact
coincides with the definition in \citep{MR2256622} of Barreira and
Iommi, who considered a mixing base space $\left(\Sigma,\sigma\right)$,
a height function $\tau$ bounded away from zero, and a function $g:Y\rightarrow\R$
such that $\Delta_{g}$ is Hölder continuous. That is\[
\mathbf{P}\left(g\mid\Phi\right)=\inf\left\{ \beta\in\mathbb{R}:\mathcal{P}_{1}\left(\Delta_{g}-\beta\tau,\Sigma_{a}^{\mathrm{per}}\right)\le0\right\} .\]
As a consequence of Theorem \ref{thm:taupressure_is_entropy} and
Corollary \ref{cor:ExhaustGurevich} we obtain that the above formula
also holds true for the case where $A$ is only irreducible, and $\tau$
is not necessarily bounded away from zero and --~different from Savchenko~--
may depend on more than one coordinate. 
\end{rem}
Exhausting principles, group extensions, and amenability

Another important feature of the induced pressure is the possibility
to restrict appropriately to certain collections $\mathcal{C}\subset\Sigma^{*}$
of finite words. As we will discuss next, this provides us with a
powerful tool to study for instance the structure of infinite group
extensions. 

Let us now focus on the  relations for different choices of subcollections
$\mathcal{C}$, $\mathcal{C}'\subset\Sigma^{*}$. Recall from \citep{MR1732376,MR1853808}
that a sufficient condition for $\mathcal{P}_{1}\left(\varphi,\Sigma_{a}^{\mathrm{per}}\right)$
and $\mathcal{P}_{1}\left(\varphi,\Sigma^{*}\right)$ to coincide
is that the Markov shift is finitely irreducibe. As a general principle
we will show in Corollary \ref{cor:urbgur-criterion} that \[
\mathcal{P}_{\psi}\left(\varphi,\Sigma^{*}\right)=\mathcal{P}_{\psi}\left(\varphi,\Sigma_{a}^{\mathrm{per}}\right)\iff\left(\psi,\varphi,\Sigma^{*}\right)\,\mbox{satisfies the exhausting principle}.\]
We say that the \emph{exhausting principle} holds for $\left(\psi,\varphi,\mathcal{C}\right)$,
if there exists a sequence $\left(K_{n}\right)_{n\in\N}$ of compact
$\sigma$-invariant subsets of $\Sigma$, such that \[
\lim_{n\to\infty}\mathcal{P}_{\psi,K_{n}}\left(\varphi,\mathcal{C}\right)=\mathcal{P}_{\psi}\left(\varphi,\mathcal{C}\right),\:\mbox{ where }\mathcal{P}_{\psi,K}\left(\varphi,\mathcal{C}\right):=\mathcal{P}_{\psi\big|_{K}}\left(\varphi\big|_{K},\mathcal{C}\cap K^{*}\right),\]
with $K^{*}:=\left\{ \omega\in\Sigma^{*}:\left[\omega\right]\cap K\neq\emptyset\right\} .$
A sufficient condition for the exhausting principle to hold for $\left(\psi,\varphi,\Sigma^{*}\right)$
and arbitrary Hölder continuous functions $\psi>0$ and $\varphi$
is that $A$ is finitely irreducible (see Corollary \ref{cor:exhausting-if-finitely-irred})
and hence this fact gives a generalisation of the observation in \citep[Theorem 2.1.5]{MR2003772}
to our context. However, Example \ref{exa:renewalshift} will show
that this condition is not necessary.

In fact, the above observation is only a special case of a more general
setting. For this we consider collections $\mathcal{C}'\subset\mathcal{C}\subset\Sigma^{*}$,
which are representable by so-called $\mathcal{C}$-loops (see Definition
\ref{def:representablebyloops}). See Theorem \ref{thm:exhaustable-part}
for the corresponding main results. One of our purposes to introduce
this general framework is to study skew product dynamical systems
connected to infinite group extensions. More precisely, we consider
certain shift spaces involving the structure of a finitely generated
group $G:=\mathbb{F}_{k}/N$, where $\mathbb{F}_{k}$ denotes the
free group with $k\in\N$ generators. For a  natural choice of $\mathcal{C}'\subset\mathcal{C}$
of finite words we obtain in Theorem \ref{thm:-Group-Extension} that
\begin{eqnarray*}
G\,\mbox{is amenable} & \iff & \left(1,0,\mathcal{C}\right)\,\mbox{satisfies the exhausting principle}\\
 & \iff & \mathcal{P}_{1}\left(0,\mathcal{C}\right)=\mathcal{P}_{1}\left(0,\mathcal{C}'\right).\end{eqnarray*}
Moreover, we have that  $\mathcal{C}$ is finitely irreducible (see
Definition \ref{def:Finitely_irreducible_C}), if and only if the
group $G$ is finite. We would like to point out that Theorem \ref{thm:-Group-Extension}
provides us with a large class of interesting examples for which the
exhausting principle holds, while the underlying Markov shift is far
from being finitely irreducible. In Example \ref{exa:Z-extension}
we elaborate this idea explicitly.

\section{Basic properties of the induced pressure }

In this section we will investigate the basic properties of the $\psi$--induced
pressure for non-negative scaling functions $\psi:\Sigma\to\R$ and
subsets $\mathcal{C}\subset\Sigma^{*}$. 

We call a function $f:\Sigma\rightarrow\R$ $\alpha$-Hölder continuous,
if there exists $\alpha>0$ and a constant $V_{\alpha}\left(f\right)$,
such that for all $\omega,\omega'\in\Sigma$ \begin{equation}
\left|f\left(\omega\right)-f\left(\omega'\right)\right|\le e^{-\alpha}V_{\alpha}\left(f\right)d_{\alpha}\left(\omega,\omega'\right).\label{eq:f-hoeldercontinuous-def}\end{equation}
We call $f$ Hölder continuous, if there exists $\alpha>0$ such that
$f$ is $\alpha$-Hölder continuous.

The bounded distortion lemma (see e.g.\ \citep[Lemma 2.3.1]{MR2003772})
shows that for a Hölder continuous function $f$ we have for every
$\omega\in\Sigma^{*}$ and $\tau,\tau'\in\mbox{\ensuremath{\left[\omega\right]}}$\begin{equation}
\left|S_{\left|\omega\right|}f\left(\tau\right)-S_{\left|\omega\right|}f\left(\tau'\right)\right|\le\frac{V_{\alpha}\left(f\right)}{e^{\alpha}-1}d_{\alpha}\left(\sigma^{\left|\omega\right|}\left(\tau\right),\sigma^{\left|\omega\right|}\left(\tau'\right)\right).\label{eq:boundeddistortion-lemma}\end{equation}
This in particular implies the existence of a constant $C_{f}>0$,
such that \begin{equation}
\left|S_{\left|\omega\right|}f\left(\tau\right)-S_{\left|\omega\right|}f\left(\tau'\right)\right|\le C_{f}\label{eq:boundeddistortionproperty}\end{equation}
for all $\omega\in\Sigma^{*}$ and $\tau,\tau'\in\left[\omega\right]$.
We will refer to (\ref{eq:boundeddistortionproperty}) as the \emph{bounded
distortion property}. 
\begin{rem}
\label{rem:def_forhoelderpotentials} If $\varphi$ and $\psi$ are
both Hölder continuous and $\mathcal{C}\subset\Sigma^{*}$, then in
the definition of $\mathcal{P}_{\psi}\left(\varphi,\mathcal{C}\right)$
we may replace $S_{\omega}\varphi$ by $\inf_{\tau\in\left[\omega\right]}\sum_{k=0}^{|\omega|-1}\varphi\circ\sigma^{k}\left(\tau\right)$
without altering the value of the pressure. If, additionally, $\mathcal{C}:=\Sigma^{\mathrm{per}}$
or $\Sigma_{a}^{\mathrm{per}}$, $a\in I$, then we may also use $\sum_{k=0}^{|\omega|-1}\varphi\circ\sigma^{k}\left(\overline{\omega}\right)$
for $\omega\in\Sigma^{\mathrm{per}}$. Independent of the particular
choice for $S_{\omega}\varphi$ we may also choose $S_{\omega}\psi$
in an analogue fashion. 

Also notice, that, if $\varphi$ and $\psi$ are both Hölder continuous,
then we have for the same reason $\mathcal{P}_{\psi,K}\left(\varphi,\mathcal{C}\right)=\mathcal{P}_{\psi}\left(\varphi,\mathcal{C}\cap K^{*}\right)$
for all $\sigma$-invariant subsets $K\subset\Sigma$.
\end{rem}
The following properties of the induced pressure are similar to the
well-known properties of the classical topological pressure (cf.\
\citep[Sec. 9.2]{MR648108}). 
\begin{prop}
\label{fac:pressure}Let $\varphi,\varphi_{1},\varphi_{2},\psi,\psi_{1},\psi_{2}$
be real functions on $\Sigma$. The induced topological pressure has
the following basic properties.
\begin{enumerate}
\item \label{enu:(Monotonicity)}\emph{(Monotonicity)} For $\varphi_{1}\le\varphi_{2}$,
$\psi_{2}\ge\psi_{1}\ge0$, $\mathcal{C}_{1}\subset\mathcal{\mathcal{C}}_{2}\subset\Sigma^{*}$,
and $K_{1}\subset K_{2}\subset\Sigma$ $\sigma$-invariant, we have,
\[
\mathcal{P}_{\psi_{1},K_{1}}\left(\varphi_{1},\mathcal{C}_{1}\right)\le\mathcal{P}_{\psi_{2},K_{2}}\left(\varphi_{2},\mathcal{\mathcal{C}}_{2}\right).\]

\item \label{enu:(Continuity)}\emph{(Continuity)} If $\psi\ge c>0$, or
$\psi>0$ and $\varphi$ are both Hölder continuous and $\mathcal{C}=\Sigma^{*}$,
then \[
\mathcal{P}_{\left(\cdot\right)}\left(\cdot,\mathcal{C}\right):C\left(\Sigma,\R\right)\times C\left(\Sigma,\R\right)\rightarrow\overline{\R}\]
 is continuous in $\left(\varphi,\psi\right)$ with respect to the
product topology, where the set of continuous functions $C\left(\Sigma,\R\right)$
is equipped with the $\Vert\cdot\Vert_{\infty}$--norm.
\item \label{enu:(Convexity)}\emph{(Convexity)} For $\psi\ge0$ and $\varphi_{1},\varphi_{2}\in C\left(\Sigma,\R\right)$
with $\mathcal{P}_{\psi}\left(\varphi_{i},\mathcal{C}\right)>-\infty$,
for $i\in\left\{ 1,2\right\} $ and $t\in\left(0,1\right)$ we have
\[
\mathcal{P}_{\psi}\left(t\varphi_{1}+\left(1-t\right)\varphi_{2},\mathcal{C}\right)\le t\mathcal{P}_{\psi}\left(\varphi_{1},\mathcal{C}\right)+\left(1-t\right)\mathcal{P}_{\psi}\left(\varphi_{2},\mathcal{C}.\right)\]

\item \label{enu:(Translation)}\emph{(Translation)} If $0<m\le\psi\le M$
then \[
\mathcal{P}_{\psi}\left(\varphi,\mathcal{C}\right)+c/M\le\mathcal{P}_{\psi}\left(\varphi+c,\mathcal{C}\right)\le\mathcal{P}_{\psi}\left(\varphi,\mathcal{C}\right)+c/m.\]
 
\item \label{enu:(Subadditivity)}\emph{(Subadditivity)} We have \[
\mathcal{P}_{\psi}\left(\varphi_{1}+\varphi_{2},\mathcal{C}\right)\le\mathcal{P}_{\psi}\left(\varphi_{1},\mathcal{C}\right)+\mathcal{P}_{\psi}\left(\varphi_{2},\mathcal{C}\right).\]
 
\item \label{enu:(Stability)}\emph{(Stability)} For $\mathcal{C}_{i}\subset\Sigma^{*}$
, $i\in\left\{ 1,\dots,n\right\} $, we have \[
\mathcal{P}_{\psi}\left(\varphi,\bigcup_{i=1}^{n}\mathcal{C}_{i}\right)=\max\left\{ \mathcal{P}_{\psi}\left(\varphi,\mathcal{C}_{i}\right):i\in\left\{ 1,\dots,n\right\} \right\} .\]

\item \textup{\emph{\label{enu:(Subhomogeneity)}}}\textup{(Subhomogeneity)}\textup{\emph{
}}\textup{If $\mathcal{P}_{\psi}\left(\varphi,\mathcal{C}\right)\in\R$
then $\mathcal{P}_{\psi}\left(c\varphi,\mathcal{C}\right)\le c\mathcal{P}_{\psi}\left(\varphi,\mathcal{C}\right)$,
for $c\ge1$, and $\mathcal{P}_{\psi}\left(c\varphi,\mathcal{C}\right)\ge c\mathcal{P}_{\psi}\left(\varphi,\mathcal{C}\right)$,
for $c\le1$. }
\item \label{enu:(Bounded-cocycle)}\emph{(Bounded cocycle)} Let us write
$\varphi_{1}\sim\varphi_{2}$ for $\varphi_{1},\varphi_{2}\in C\left(\Sigma,\R\right)$,
if $\varphi_{1}$ and $\varphi_{2}$ are cohomologous in the class
of bounded continuous functions, i.e. there exists a bounded function
$h\in C\left(\Sigma,\R\right)$ such that $\varphi_{1}=\varphi_{2}+h-h\circ\sigma$.
For $\varphi_{1}\sim\varphi_{2}$, we have\[
\mathcal{P}_{\psi}\left(\varphi_{1},\mathcal{C}\right)=\mathcal{P}_{\psi}\left(\varphi_{2},\mathcal{C}\right),\]
if $\varphi\sim\psi$ then we have, for every $t\in\R$, \[
\mathcal{P}_{\psi}\left(t\varphi,\mathcal{C}\right)=t+\mathcal{P}_{\psi}\left(0,\mathcal{C}\right).\]

\end{enumerate}
\end{prop}
\begin{proof}
The proofs of (\ref{enu:(Translation)}), (\ref{enu:(Subadditivity)})
and (\ref{enu:(Subhomogeneity)}) are straightforward generalisations
of the corresponding statements for the classical pressure (cf.\
\citep[Sec. 9.2]{MR648108}). (\ref{enu:(Monotonicity)}), (\ref{enu:(Stability)})
and (\ref{enu:(Bounded-cocycle)}) follow immediately from the definition
of the $\psi$--induced pressure. Hence, let us only  address (\ref{enu:(Continuity)})
and (\ref{enu:(Convexity)}). 

ad (\ref{enu:(Continuity)}): By Corollary \ref{cor:pressureformula_ifstreactlydescreasing}
below we have \[
\mathcal{P}_{\psi}\left(\varphi,\mathcal{C}\right)=\inf\left\{ \beta:\mathcal{P}_{1}\left(\varphi-\beta\psi,\mathcal{C}\right)\le0\right\} =\sup\left\{ \beta:\mathcal{P}_{1}\left(\varphi-\beta\psi,\mathcal{C}\right)\ge0\right\} .\]
The assertion follows, since $\mathcal{P}_{1}\left(\cdot,\mathcal{C}\right):C\left(\Sigma,\R\right)\rightarrow\overline{\R}$
is continuous with respect to the $\Vert\cdot\Vert_{\infty}$-norm.

ad (\ref{enu:(Convexity)}): The assumption $\mathcal{P}_{\psi}\left(\varphi_{i},\mathcal{C}\right)>-\infty$
for $i\in\left\{ 1,2\right\} $ implies that there exists a sequence
$\left(T_{k}\right)$ tending to infinity such that the sets $\left\{ \omega\in\mathcal{C}:T_{k}-\eta<S_{\omega}\psi\le T_{k}\right\} $
are non empty and \[
\lim_{k\to\infty}\frac{1}{T_{k}}\log\!\!\!\!\!\!\sum_{{\omega\in\mathcal{C}\atop T_{k}-\eta<S_{\omega}\psi\le T_{k}}}\!\!\!\!\!\!\!\negthinspace\exp S_{\omega}\left(t\varphi_{1}+\left(1-t\right)\varphi_{2}\right)=\mathcal{P}_{\psi}\left(t\varphi_{1}+\left(1-t\right)\varphi_{2},\mathcal{C}\right).\]
By Hölder's inequality we have \begin{eqnarray*}
\sum_{{\omega\in\mathcal{C}\atop T_{k}-\eta<S_{\omega}\psi\le T_{k}}}\e^{S_{\omega}\left(t\varphi_{1}+\left(1-t\right)\varphi_{2}\right)} & \le & \sum_{{\omega\in\mathcal{C}\atop T_{k}-\eta<S_{\omega}\psi\le T_{k}}}\e^{S_{\omega}t\varphi_{1}}\e^{S_{\omega}\left(1-t\right)\varphi_{2}}\\
 & \le & \left(\sum_{{\omega\in\mathcal{C}\atop T_{k}-\eta<S_{\omega}\psi\le T_{k}}}\!\!\!\!\!\!\!\!\e^{S_{\omega}\varphi_{1}}\right)^{t}\left(\sum_{{\omega\in\mathcal{C}\atop T_{k}-\eta<S_{\omega}\psi\le T_{k}}}\!\!\!\!\!\!\!\!\e^{S_{\omega}\varphi_{2}}\right)^{1-t}.\end{eqnarray*}
Taking logarithm, dividing by $T_{k}$ completes the proof.\end{proof}
\begin{rem}
For the multifractal analysis in the setting of conformal iterated
function systems the free energy function studied in \citep{JaerischKessebohmer:09}
coincides with $\beta\mapsto\mathcal{P}_{-\zeta}\left(\beta\psi,\Sigma^{*}\right)$
by Corollary \ref{cor:pressure_rescaling} below, where $\zeta$ denotes
the geometric potential associated to the conformal iterated function
system and $\psi$ is the potential defining the level sets under
consideration. The Legendre transform of the free energy function
describes the dimension spectrum of the multifractal level sets. Proposition
\ref{fac:pressure} (\ref{enu:(Convexity)}) in particular implies
that the free energy function is convex, see also \citep[Lemma 3.1]{JaerischKessebohmer:09}.
\end{rem}
The main result in this section is that the $\psi$--induced pressure
coincides with a certain critical exponent of the partition function
as stated in the following theorem. We remark that \citep[Theorem 2.1.3]{MR2003772}
is covered by our theorem by choosing $\psi=1$ (see Remark \ref{rem:criticalexponent-without-limsup}
for further comments). We would like to remark that a similar connection
has been considered by Przytycki in \citep{MR1615954} to introduce
the notion of Poincaré exponent for rational functions.
\begin{thm}
[Critical exponent] \label{thm:pressure_as_inf} For potential functions
$\varphi,\psi:\Sigma\rightarrow\R$ satisfying $\psi\ge0$ we hav\textup{e\[
\mathcal{P}_{\psi}\left(\varphi,\mathcal{C}\right)=\inf\left\{ \beta\in\R:\limsup_{T\rightarrow\infty}\sum_{{\omega\in\mathcal{C}\atop T<S_{\omega}\psi}}\exp\left(S_{\omega}\varphi-\beta S_{\omega}\psi\right)<\infty\right\} .\]
}In\textup{\emph{ particular, the definition of $\mathcal{P}_{\psi}\left(\varphi,\mathcal{C}\right)$
is independent of the choice of $\eta>0$.}}\end{thm}
\begin{proof}
Fix $\eta>0$. For $\omega\in\mathcal{C}$ let $n\left(\omega\right)$
denote the unique $n\in\N$, such that $\left(n-1\right)\eta<S_{\omega}\psi\left(\omega\right)\le n\eta$.
Observing that $\e^{-\beta n\left(\omega\right)\eta}\e^{-\left|\beta\right|\eta}\le\e^{-\beta S_{\omega}\psi}\le\e^{-\beta n\left(\omega\right)\eta}\e^{\left|\beta\right|\eta}$
for all $\omega\in\mathcal{C}$ we conclude that \[
\limsup_{T\rightarrow\infty}\sum_{{\omega\in\mathcal{C}\atop T<S_{\omega}\psi}}\exp\left(S_{\omega}\varphi-\beta S_{\omega}\psi\right)<\infty,\]
if and only if\[
\limsup_{N\rightarrow\infty}\sum_{n\ge N}\sum_{{\omega\in\mathcal{C}\atop \left(n-1\right)\eta<S_{\omega}\psi\le n\eta}}\exp\left(S_{\omega}\varphi-\beta n\eta\right)<\infty.\]
 Hence, it will be sufficient to verify that \[
\mathcal{P}_{\psi}\left(\varphi,\mathcal{C}\right)=\inf\left\{ \beta\in\R:\limsup_{N\rightarrow\infty}\sum_{n\ge N}\sum_{{\omega\in\mathcal{C}\atop \left(n-1\right)\eta<S_{\omega}\psi\le n\eta}}\exp\left(S_{\omega}\varphi-\beta n\eta\right)<\infty\right\} .\]
For the {}``$\le$'' part we will show that for every $\beta\in\R$
and $\epsilon>0$ such that $\beta<\mathcal{P}_{\psi}\left(\varphi,\mathcal{C}\right)-\epsilon$,
we have for all $N\in\N$, \begin{equation}
\sum_{n\ge N}\sum_{{\omega\in\mathcal{C}\atop \left(n-1\right)\eta<S_{\omega}\psi\le n\eta}}\exp\left(S_{\omega}\varphi-\beta n\eta\right)=\infty.\label{eq:criticalexponent-divergence}\end{equation}
By the definition of $\mathcal{P}_{\psi}\left(\varphi,\mathcal{C}\right)$
there exists a sequence $\left(T_{j}\right)_{j\in\N}$, such that
for every $j\in\N$, $T_{j+1}-T_{j}>\eta$ and \[
\frac{1}{T_{j}}\lg\sum_{{\omega\in\mathcal{C}\atop T_{j}-\eta<S_{\omega}\psi\le T_{j}}}\exp\left(S_{\omega}\varphi\right)\ge\beta+\epsilon.\]
For $\omega\in\mathcal{C}$ with $T_{j}-\eta<S_{\omega}\psi\le T_{j}$
we have $\left|n\left(\omega\right)\eta-T_{j}\right|<2\eta$ and hence
\begin{eqnarray*}
\sum_{n\ge N}\!\!\!\sum_{{\omega\in\mathcal{C}\atop \left(n-1\right)\eta<S_{\omega}\psi\left(\omega\right)\le n\eta}}\!\!\!\!\!\!\!\!\!\e^{S_{\omega}\varphi-\beta n\eta} & \ge & \sum_{j\in\N:T_{j}\ge N\eta}\sum_{{\omega\in\mathcal{C}\atop T_{j}-\eta<S_{\omega}\psi\le T_{j}}}\!\!\!\exp\left(S_{\omega}\varphi-\beta n\left(\omega\right)\eta\right),\\
 & \ge & \e^{-2\eta|\beta|}\!\!\!\sum_{j\in\N:T_{j}\ge N\eta}\sum_{{\omega\in\mathcal{C}\atop T_{j}-\eta<S_{\omega}\psi\le T_{j}}}\!\!\!\exp\left(S_{\omega}\varphi-\beta T_{j}\right)\\
 & \ge & \e^{-2\eta|\beta|}\sum_{j\in\N:T_{j}\ge N\eta}\exp\left(-\beta T_{j}\right)\exp\left(T_{j}\left(\beta+\epsilon\right)\right)\\
 & = & \infty.\end{eqnarray*}
This argument is not only valid for $\mathcal{P}_{\psi}\left(\varphi,\mathcal{C}\right)\in\R$,
but also for $\mathcal{P}_{\psi}\left(\varphi,\mathcal{C}\right)=\infty$,
in which case (\ref{eq:criticalexponent-divergence}) holds for every
$\beta\in\R$. 

For the {}``$\geq$'' part we first consider the case $\mathcal{P}_{\psi}\left(\varphi,\mathcal{C}\right)>-\infty$
and show that for every $\epsilon>0$ \[
\limsup_{N\rightarrow\infty}\sum_{n\ge N}\sum_{{\omega\in\mathcal{C}\atop \left(n-1\right)\eta<S_{\omega}\psi\le n\eta}}\exp\left(S_{\omega}\varphi-\left(\mathcal{P}_{\psi}\left(\varphi,\mathcal{C}\right)+\epsilon\right)n\eta\right)<\infty.\]
Again, by the definition of $\mathcal{P}_{\psi}\left(\varphi,\mathcal{C}\right)$
we find for all $\epsilon>0$ an element $N\in\N$, such that for
all $n\ge N$ \begin{equation}
\sum_{{\omega\in\mathcal{C}\atop \left(n-1\right)\eta<S_{\omega}\psi\le n\eta}}\exp S_{\omega}\varphi\le\exp\left(n\eta\left(\mathcal{P}_{\psi}\left(\varphi,\mathcal{C}\right)+\epsilon/2\right)\right).\label{eq:criticalexponent-upperbound}\end{equation}
Consequently, we have \begin{eqnarray*}
 &  & \!\!\!\!\!\!\!\!\!\!\!\!\!\negthinspace\!\!\!\!\!\!\!\!\!\!\!\!\!\negthinspace\!\!\!\!\!\!\!\!\!\!\sum_{n\ge N}\sum_{{\omega\in\mathcal{C}\atop \left(n-1\right)\eta<S_{\omega}\psi\le n\eta}}\exp\left(S_{\omega}\varphi-\left(\mathcal{P}_{\psi}\left(\varphi,\mathcal{C}\right)+\epsilon\right)n\eta\right)\\
 & \le & \sum_{n\ge N}\exp\left(-n\eta\left(\mathcal{P}_{\psi}\left(\varphi,\mathcal{C}\right)+\epsilon\right)\right)\exp\left(n\eta\left(\mathcal{P}_{\psi}\left(\varphi,\mathcal{C}\right)+\epsilon/2\right)\right)\\
 & = & \sum_{n\ge N}\exp\left(-n\eta\epsilon/2\right)<\infty.\end{eqnarray*}
For the case $\mathcal{P}_{\psi}\left(\varphi,\mathcal{C}\right)=-\infty$
we have for all $\rho\in\R$ and all $n\in\mathbb{N}$ sufficiently
large that $\sum_{{\omega\in\mathcal{C}\atop \left(n-1\right)\eta<S_{\omega}\psi\le n\eta}}\exp S_{\omega}\varphi\le\exp\left(n\eta\rho\right)$.
Hence, by the same arguments we conclude that for all $\beta\in\R$
and $N$ sufficiently large we have convergence of the series \[
\sum_{n\ge N}\sum_{{\omega\in\mathcal{C}\atop \left(n-1\right)\eta<S_{\omega}\psi\le n\eta}}\exp\left(S_{\omega}\varphi-\beta n\eta\right).\]
 \end{proof}
\begin{rem}
The previous proof in particular shows that it is sufficient to take
the limes superior with respect to the subsequence $T_{k}=\eta k$,
$k\in\N$, $\eta>0$. More precisely, for potential functions $\varphi,\psi:\Sigma\rightarrow\R$
with $\psi\ge0$ and for any $\eta>0$ we have \begin{eqnarray*}
\mathcal{P}_{\psi}\left(\varphi,\mathcal{C}\right) & = & \limsup_{n}\frac{1}{n\eta}\log\sum_{{\omega\in\mathcal{C}\atop \left(n-1\right)\eta<S_{\omega}\psi\le n\eta}}\exp S_{\omega}\varphi.\end{eqnarray*}

\end{rem}
The next Lemma gives an alternative way to describe the $\psi$--induced
pressure, which will prove to be useful subsequently.
\begin{lem}
\label{lem:pressure-alternativedef}Let $\varphi,\psi:\Sigma\rightarrow\R$
be Hölder continuous and $\psi\ge0$ and $\mathcal{C}\subset\Sigma^{*}$.
If $\mathcal{P}_{\psi}\left(\varphi,\mathcal{C}\right)\ge0$ then
for all sufficiently large $T>0$ we have \[
\mathcal{P}_{\psi}\left(\varphi,\mathcal{C}\right)=\limsup_{n\rightarrow\infty}\frac{1}{n}\log\sum_{{\omega\in\mathcal{C}\atop T<S_{\omega}\psi\le n}}\exp S_{\omega}\varphi,\]
and for $\mathcal{P}_{\psi}\left(\varphi,\mathcal{C}\right)<0$ we
have \[
\mathcal{P}_{\psi}\left(\varphi,\mathcal{C}\right)=\limsup_{n\rightarrow\infty}\frac{1}{n}\log\sum_{{\omega\in\mathcal{C}\atop n\le S_{\omega}\psi}}\exp S_{\omega}\varphi.\]
\end{lem}
\begin{proof}
By the definition of $\mathcal{P}_{\psi}\left(\varphi,\mathcal{C}\right)$
we find for $\epsilon>0$ an element $N\in\N$, such that for all
$n\ge N$ \begin{equation}
\sum_{{\omega\in\mathcal{C}\atop n-1<S_{\omega}\psi\le n}}\exp S_{\omega}\varphi\le\exp\left(n\left(\mathcal{P}_{\psi}\left(\varphi,\mathcal{C}\right)+\epsilon\right)\right).\label{eq:LimsupBoundForPressure}\end{equation}
Let us begin with the case $\mathcal{P}_{\psi}\left(\varphi,\mathcal{C}\right)\ge0$.
Then (\ref{eq:LimsupBoundForPressure}) implies \[
\sum_{{\omega\in\mathcal{C}\atop N-1<S_{\omega}\psi\le n}}\!\!\!\e^{S_{\omega}\varphi}\le\sum_{k=N-1}^{n}\e^{k\left(\mathcal{P}_{\psi}\left(\varphi,\mathcal{C}\right)+\epsilon\right)}\le\frac{\exp\left(\left(n+1\right)\left(\mathcal{P}_{\psi}\left(\varphi,\mathcal{C}\right)+\epsilon\right)\right)-1}{\exp\left(\mathcal{P}_{\psi}\left(\varphi,\mathcal{C}\right)+\epsilon\right)-1}.\]
Since $\mathcal{P}_{\psi}\left(\varphi,\mathcal{C}\right)+\epsilon>0$
we may take logarithm and divide by $n$ to obtain the\textbf{ }upper
bound. 

For the case $\mathcal{P}_{\psi}\left(\varphi,\mathcal{C}\right)<0$
fix $\epsilon\in\left(0,-\mathcal{P}_{\psi}\left(\varphi,\mathcal{C}\right)\right)$.
This time (\ref{eq:LimsupBoundForPressure}) implies for $n\ge N$
\[
\sum_{{\omega\in\mathcal{C}\atop n\le S_{\omega}\psi}}\!\!\!\e^{S_{\omega}\varphi}\le\sum_{k=n}^{\infty}\e^{k\left(\mathcal{P}_{\psi}\left(\varphi,\mathcal{C}\right)+\epsilon\right)}=\frac{\exp\left(n\left(\mathcal{P}_{\psi}\left(\varphi,\mathcal{C}\right)+\epsilon\right)\right)}{1-\exp(\mathcal{P}_{\psi}\left(\varphi,\mathcal{C}\right)+\epsilon)},\]
and hence taking logarithm and dividing by $n$ finishes the proof
for the\textbf{ }upper bound. The lower bound is immediate in both
cases.\end{proof}
\begin{rem}
\label{rem:criticalexponent-without-limsup}We would like to compare
Theorem \ref{thm:pressure_as_inf} with \citep[Theorem 2.1.3]{MR2003772}
and also the classical definition of the Poincaré exponent. Under
the condition that $\limsup_{n}\sum_{{\omega\in\mathcal{C}\atop n-1<S_{\omega}\psi\le n}}\e^{S_{\omega}\varphi}<\infty$
implies that $\sum_{{\omega\in\mathcal{C}\atop n-1<S_{\omega}\psi\le n}}\e^{S_{\omega}\varphi}<\infty$
for all $n\in\N$ we have by Theorem \ref{thm:pressure_as_inf} \begin{equation}
\mathcal{P}_{\psi}\left(\varphi,\Sigma^{*}\right)=\inf\left\{ \beta\in\R:\sum_{\omega\in\Sigma^{*}}\exp\left(S_{\omega}\varphi-\beta S_{\omega}\psi\right)<\infty\right\} \label{eq:pressure-via-poincare}\end{equation}
and by Lemma \ref{lem:pressure-alternativedef}, for $\mathcal{P}_{\psi}\left(\varphi,\mathcal{C}\right)\ge0$,
\[
\mathcal{P}_{\psi}\left(\varphi,\mathcal{C}\right)=\limsup_{n\rightarrow\infty}\frac{1}{n}\log\sum_{{\omega\in\mathcal{C}\atop S_{\omega}\psi\le n}}\exp S_{\omega}\varphi.\]
Generalising the proof of \citep[Proposition 2.1.9]{MR2003772} this
is in particular the case for $A$ finitely irreducible, $\mathcal{C}=\Sigma^{*}$,
and $\psi\ge c>0$. In general, the identity (\ref{eq:pressure-via-poincare})
does not hold. For instance, consider the renewal shift with state
space $\N$ rooted at $1$, which is mixing but not finitely irreducible.
Next choose the potential functions $\varphi_{\left[k\right]}:=-\log k$
for $k\ge2$ and $\varphi_{|\left[1,n\right]}:=-2\log n$ for $n\ge1$.
Then $\mathcal{P}_{1}\left(\varphi,\Sigma^{*}\right)\le\frac{1}{2}\log\sum_{\left|\omega\right|=2}\e^{S_{\omega}\varphi}=\frac{1}{2}\log\left(\sum_{k\ge2}\e^{S_{\left(k,k-1\right)}\varphi}+\sum_{n\ge1}\e^{S_{\left(1,n\right)}\varphi}\right)=\frac{1}{2}\log\left(\sum_{k\ge2}1/k\left(k-1\right)+\sum_{n\ge1}1/n^{2}\right)<\infty$,
whereas for all $\beta\in\R$ we have $\sum_{\omega\in\Sigma^{*}}\e^{S_{\omega}\varphi-\beta\left|\omega\right|}\ge\sum_{k\ge2}1/k=\infty$
implying that in this situation the right hand side of (\ref{eq:pressure-via-poincare})
is equal to infinity. \end{rem}
\begin{cor}
\label{cor:pressure_rescaling} For potential functions $\varphi,\psi:\Sigma\rightarrow\R$,
$\psi\geq0$, we always have \begin{equation}
\mathcal{P}_{\psi}\left(\varphi,\mathcal{C}\right)\ge\inf\left\{ \beta\in\R:\mathcal{P}_{1}\left(\varphi-\beta\psi,\mathcal{C}\right)\le0\right\} .\label{eq:upperbound}\end{equation}
\end{cor}
\begin{proof}
Since\begin{eqnarray*}
\inf\left\{ \beta\in\R:\limsup_{T\rightarrow\infty}\sum_{{\omega\in\mathcal{C}\atop T<S_{\omega}\psi}}\exp\left(S_{\omega}\varphi-\beta S_{\omega}\psi\right)<\infty\right\} \\
 &  & \!\!\!\!\!\!\!\!\!\!\!\!\!\!\!\!\!\!\!\!\!\!\!\!\!\!\!\!\!\!\!\!\!\!\geq\inf\left\{ \beta\in\R:\mathcal{P}_{1}\left(\varphi-\beta\psi,\mathcal{C}\right)\le0\right\} \end{eqnarray*}
the inequality (\ref{eq:upperbound}) follows by Theorem \ref{thm:pressure_as_inf}.
\end{proof}
We say that $\mathcal{P}_{\psi}\left(\varphi,\mathcal{C}\right)$
is given by the pseudo inverse (of the $1$--induced pressure), if
$\mathcal{P}_{\psi}\left(\varphi,\mathcal{C}\right)=\inf\left\{ \beta\in\R:\mathcal{P}_{1}\left(\varphi-\beta\psi,\mathcal{C}\right)\le0\right\} $. 
\begin{rem}
In general only {}``$\ge$'' holds as shown in Corollary \ref{cor:pressure_rescaling}.
In the following trivial example we have {}``$>$''. Consider $\Sigma$
with state space $\N$, incidence matrix $A_{ij}=\delta_{ij}$, and
the potential $\psi\left(\omega\right)=\log\left(\omega_{1}\right)$.
Then we have $\mathcal{P}_{\psi}\left(0,\Sigma^{*}\right)=\infty$,
whereas $\inf\left\{ \beta\in\R:\mathcal{P}_{1}\left(-\beta\psi,\Sigma^{*}\right)\le0\right\} =1$. 
\end{rem}
In the next two corollaries we give sufficient conditions for the
$\psi$--induced pressure to be given by the pseudo-inverse of the
$1$--induced pressure.
\begin{cor}
\label{cor:pressureformula_ifstreactlydescreasing} For potential
functions $\varphi,\psi:\Sigma\rightarrow\R$ with $\psi\ge0$ , such
that the map $\beta\mapsto\mathcal{P}_{1}\left(\varphi-\beta\psi,\mathcal{C}\right)$
is strictly decreasing on $\interior\left\{ \beta:\mathcal{P}_{1}\left(\varphi-\beta\psi,\mathcal{C}\right)<\infty\right\} $
we have \[
\mathcal{P}_{\psi}\left(\varphi,\mathcal{C}\right)=\inf\left\{ \beta\in\R:\mathcal{P}_{1}\left(\varphi-\beta\psi,\mathcal{C}\right)\le0\right\} =\sup\left\{ \beta\in\R:\mathcal{P}_{1}\left(\varphi-\beta\psi,\mathcal{C}\right)\ge0\right\} .\]
In particular, this is the case if $\psi\ge c>0$. Another condition
is also satisfied if $\mathcal{C}=\Sigma^{*}$, the potentials $\varphi$,
$\psi$ are Hölder continuous, the incidence matrix of $\Sigma$ is
finitely irreducible and $\psi^{-1}\left(\left\{ 0\right\} \right)$
is at most countable. \end{cor}
\begin{proof}
Since the map $\beta\mapsto\mathcal{P}_{1}\left(\varphi-\beta\psi,\mathcal{C}\right)$
is strictly decreasing on the set $\interior\left\{ \beta:\mathcal{P}_{1}\left(\varphi-\beta\psi,\mathcal{C}\right)<\infty\right\} $,
we conclude that \begin{eqnarray*}
\inf\left\{ \beta\in\R:\mathcal{P}_{1}\left(\varphi-\beta\psi,\mathcal{C}\right)\le0\right\}  & = & \inf\left\{ \beta\in\R:\mathcal{P}_{1}\left(\varphi-\beta\psi,\mathcal{C}\right)<0\right\} \\
 & = & \sup\left\{ \beta\in\R:\mathcal{P}_{1}\left(\varphi-\beta\psi,\mathcal{C}\right)\ge0\right\} .\end{eqnarray*}
Since \begin{eqnarray*}
\inf\left\{ \beta\in\R:\mathcal{P}_{1}\left(\varphi-\beta\psi,\mathcal{C}\right)<0\right\} \\
 &  & \!\!\!\!\!\!\!\!\!\!\!\!\!\!\!\!\!\!\!\!\!\!\!\!\!\!\!\!\!\!\!\!\geq\inf\left\{ \beta\in\R:\limsup_{T\rightarrow\infty}\sum_{{\omega\in\mathcal{C}\atop T<S_{\omega}\psi}}\exp\left(S_{\omega}\varphi-\beta S_{\omega}\psi\right)<\infty\right\} \end{eqnarray*}
the claim follows from Theorem \ref{thm:pressure_as_inf} and Corollary
\ref{cor:pressure_rescaling}. 

For $\psi\ge c>0$ the map $\beta\mapsto\mathcal{P}_{1}\left(\varphi-\beta\psi,\mathcal{C}\right)$
is  strictly decreasing on the set $\interior\left\{ \beta:\mathcal{P}_{1}\left(\varphi-\beta\psi,\mathcal{C}\right)<\infty\right\} $.
If $\varphi$ and $\psi$ are Hölder continuous, $\psi\geq0$ and
the incidence matrix of $\Sigma$ is finitely irreducible, then we
have by \citep{MR2439479} that $\beta\mapsto\mathcal{P}_{1}\left(\varphi-\beta\psi,\Sigma^{*}\right)$
is real analytic on $\interior\left\{ \beta:\mathcal{P}_{1}\left(\varphi-\beta\psi,\Sigma^{*}\right)<\infty\right\} $
and \[
\frac{\partial}{\partial\beta}\mathcal{P}_{1}\left(\varphi-\beta\psi,\Sigma^{*}\right)\big|_{\beta_{0}}=-\int\psi d\mu_{\varphi-\beta_{0}\psi}<0,\]
where $\mu_{\varphi-\beta_{0}\psi}$ denotes the unique invariant
Gibbs measure of $\varphi-\beta_{0}\psi$, which has no atoms.\end{proof}
\begin{rem}
\label{rem:psipressure_finitealphabet} If $\Sigma$ is a subshift
over a finite alphabet, i.e. $\Sigma\subset\Sigma_{n}:=\left\{ 1,\ldots,n\right\} ^{\mathbb{N}}$
for some $n\in\N$, and $\psi>0$ then the map $\beta\mapsto\mathcal{P}_{1}\left(\varphi-\beta\psi,\mathcal{C}\right)$
is a strictly decreasing continuous map on $\R$. Hence we conclude
that $\mathcal{P}_{\psi}\left(\varphi,\mathcal{C}\right)$ is its
unique zero, i.e. $\mathcal{P}_{1}\left(\varphi-\mathcal{P}_{\psi}\left(\varphi,\mathcal{C}\right)\psi,\mathcal{C}\right)=0.$ 
\end{rem}
The next Corollary shows that the exhausting principle implies, that
$\mathcal{P}_{\psi}\left(\varphi,\mathcal{C}\right)$ is given by
the pseudo inverse of the $1$--induced pressure. For ease of notation
let $C_{\Sigma,\sigma}$ denote the set of compact $\sigma$-invariant
subsets of $\Sigma$. 
\begin{cor}
\label{cor:psipressure_asinf-underexhaustingprop}For potential functions
$\varphi,\psi\in C\left(\Sigma,\R\right)$ satisfying $\psi>0$ and
such that the exhausting principle holds for $\left(\psi,\varphi,\mathcal{C}\right)$
we have \[
\mathcal{P}_{\psi}\left(\varphi,\mathcal{C}\right)=\inf\left\{ \beta\in\R:\mathcal{P}_{1}\left(\varphi-\beta\psi,\mathcal{C}\right)\le0\right\} .\]
\end{cor}
\begin{proof}
By the first part of Corollary \ref{cor:pressure_rescaling} we only
have to prove {}``$\le$ {}``: For every $K\in C_{\Sigma,\sigma}$
we have that $\psi_{|K}$ is bounded away from zero and hence by the
second part of Corollary \ref{cor:pressure_rescaling} and Proposition
\ref{fac:pressure} (\ref{enu:(Monotonicity)}) we have \begin{eqnarray*}
\mathcal{P}_{\psi,K}\left(\varphi,\mathcal{C}\right) & = & \inf\left\{ \beta\in\R:\mathcal{P}_{1,K}\left(\varphi-\beta\psi,\mathcal{C}\right)\le0\right\} \\
 & \le & \inf\left\{ \beta\in\R:\mathcal{P}_{1}\left(\varphi-\beta\psi,\mathcal{C}\right)\le0\right\} .\end{eqnarray*}
Since the\emph{ }exhausting principle holds for $\left(\psi,\varphi,\mathcal{C}\right)$
we know that there exists a sequence $\left(K_{n}\right)_{n\in\N}\in\left(C_{\Sigma,\sigma}\right)^{\N}$,
such that $\lim_{n\to\infty}\mathcal{P}_{\psi,K_{n}}\left(\varphi,\mathcal{C}\right)=\mathcal{P}_{\psi}\left(\varphi,\mathcal{C}\right).$
Hence we conclude \[
\mathcal{P}_{\psi}\left(\varphi,\mathcal{C}\right)\le\inf\left\{ \beta\in\R:\mathcal{P}_{1}\left(\varphi-\beta\psi,\mathcal{C}\right)\le0\right\} .\]

\end{proof}

\section{Loop spaces\label{sec:Loop-spaces}}

In this section we introduce the $\mathcal{C}$-loop space for collections
$\mathcal{C}\subset\Sigma^{*}$ satisfying certain inducing properties. 

In the first part we motivate the construction and prove basic properties.
In the second part we investigate the question under which assumptions
we have, that two collections $\mathcal{C}'$ and \foreignlanguage{ngerman}{$\mathcal{C}$}
give rise to the same value of the induced pressure. We will show
that this question is closely linked to a certain exhausting principle.
We will introduce dynamical group extensions to apply this new insights
to characterise amenability of the underlying group in terms of the
induced pressure. The third part of this section is devoted to a detailed
comparison of the classical pressure and the Gurevi\v c pressure;
we also obtain a variational principle for the generalised Gurevi\v c
pressure.

\subsection{Construction and basic properties of loop spaces \label{sub:Construction-of-loop-spaces}}

The following definition will be crucial for the construction of loop
spaces. 
\begin{defn}
\label{def:representablebyloops}We say that $\mathcal{C}$ is \emph{closed
under concatenations}, if for any $\omega_{1},\omega_{2}\in\mathcal{C}$
satisfying $\omega_{1}\omega_{2}\in\Sigma^{*}$ we have $\omega_{1}\omega_{2}\in\mathcal{C}$. 

We say that $\mathcal{C}$ has the \emph{refinement property, }if
for $\omega_{1}\omega_{2}\in\mathcal{C}$ and $\omega_{2}\omega_{3}\in\mathcal{C}$
we have $\omega_{2}\in\mathcal{C}$. 

If $\mathcal{C}$ is both closed under concatenations and has the
refinement property then we say that $\mathcal{C}$ is \emph{representable
by loops}. 
\end{defn}
In fact, this representation is given in the following way. For a
family $\mathcal{C}\subset\Sigma^{*}$ which is representable by loops
we introduce the set of \emph{simple $\mathcal{C}$-loops} as \[
\mathcal{C}^{\mathrm{smpl}}:=\left\{ C\in\mathcal{C}:\nexists C_{1},C_{2}\in\mathcal{C}:C=C_{1}C_{2}\right\} .\]

We define the $\mathcal{C}$\emph{-loop space} \[
\widetilde{\Sigma}_{\mathcal{C}}:=\Sigma_{\mathcal{C}^{\mathrm{smpl}}}:=\left\{ \omega\in\left(\mathcal{C}^{\mathrm{smpl}}\right)^{\N}:\,\,\,\forall i\in\N:\:\omega_{i}\omega_{i+1}\in\mathcal{C}\right\} ,\]
which is a subshift of finite type over the alphabet $\mathcal{C}^{\mathrm{smpl}}$.
 Let $\widetilde{\sigma}$ denote the corresponding shift dynamic
on $\widetilde{\Sigma}_{\mathcal{C}}$ and let $\tilde{d}_{\alpha}$
denote the corresponding metric on $\widetilde{\Sigma}_{\mathcal{C}}$. 

Since $\mathcal{C}$ is closed under concatenations there is a canonical
map $\left(\widetilde{\Sigma}_{\mathcal{C}}\right)^{*}\rightarrow\mathcal{C}$
which we denote by $\widetilde{\omega}\mapsto\omega$. By definition
of the simple elements this map is surjective and since $\mathcal{C}$
has the refinement property it is also injective. 

The injection $\iota:\left(\widetilde{\Sigma}_{\mathcal{C}},\widetilde{d}_{\alpha}\right)\hookrightarrow\left(\Sigma,d_{\alpha}\right)$
is Lipschitz continuous. We will also indicate this map by omitting
the tilde, i.e. we simply write $\omega$ for $\iota\left(\widetilde{\omega}\right)$. 

For a function $f:\Sigma\rightarrow\R$ we define its \emph{induced
version} \[
\tilde{f}:\widetilde{\Sigma}_{\mathcal{C}}\rightarrow\R\textrm{ by }\tilde{f}\left(\widetilde{\omega}\right):=S_{\left|\omega_{1}\right|}f\left(\omega\right),\]
where $\widetilde{\omega}:=\left(\widetilde{\omega}_{1},\widetilde{\omega}_{2},\ldots\right)$.
Let us verify that if $f:\Sigma\rightarrow\R$ is $\alpha$-Hölder
continuous then also $\tilde{f}:\widetilde{\Sigma}_{\mathcal{C}}\rightarrow\R$
is $\alpha$-Hölder continuous. This is a consequence of (\ref{eq:boundeddistortion-lemma}),
since for $\tilde{\omega},\tilde{\omega}'\in\widetilde{\Sigma}_{\mathcal{C}}$
with $\tilde{\omega}\wedge\tilde{\omega}'\ge1$ we have \begin{eqnarray*}
\left|\tilde{f}\left(\tilde{\omega}\right)-\tilde{f}\left(\tilde{\omega}'\right)\right| & = & \left|S_{\left|\omega_{1}\right|}f\left(\omega\right)-S_{\left|\omega_{1}\right|}f\left(\omega'\right)\right|\\
 & \le & \frac{V_{\alpha}\left(f\right)}{e^{\alpha}-1}d_{\alpha}\left(\sigma^{\left|\omega_{1}\right|}\left(\omega\right),\sigma^{\left|\omega_{1}\right|}\left(\omega'\right)\right)\\
 & = & \frac{V_{\alpha}\left(f\right)}{e^{\alpha}-1}d_{\alpha}\left(\iota\left(\tilde{\sigma}\left(\tilde{\omega}\right)\right),\iota\left(\tilde{\sigma}\left(\tilde{\omega}'\right)\right)\right)\\
 & \le & \frac{V_{\alpha}\left(f\right)}{e^{\alpha}-1}\tilde{d}_{\alpha}\left(\tilde{\omega},\tilde{\omega}'\right).\end{eqnarray*}
Hence, (\ref{eq:f-hoeldercontinuous-def}) is satisfied with $V_{\alpha}\left(\tilde{f}\right):=e^{\alpha}\frac{V_{\alpha}\left(f\right)}{e^{\alpha}-1}$
and consequently $\tilde{f}$ is $\alpha$-Hölder continuous. It is
important to note that Hölder continuity of $f$ is a sufficient,
but not a necessary condition for the induced version $\tilde{f}$
to be Hölder continuous.

The next Theorem shows that the $\psi$--induced pressure is invariant
under inducing. Although straightforward to prove, it is of structural
importance for studying loop spaces.
\begin{thm}
[Invariance under Inducing]\label{thm:invariance under inducing}Let
$\varphi,\psi:\Sigma\rightarrow\R$ be Hölder continuous satisfying
$\psi>0$ and let $\mathcal{C}$ be\emph{ }representable by loops\emph{.}
Then for $\mathcal{C}'\subset\mathcal{C}$ and the canonically associated
set $\widetilde{\mathcal{C}}'\subset\left(\widetilde{\Sigma}_{\mathcal{C}}\right)^{*}$
we have \[
\mathcal{P}_{\psi}\left(\varphi,\mathcal{C}'\right)=\mathcal{P}_{\tilde{\psi}}\left(\tilde{\varphi},\widetilde{\mathcal{C}}'\right).\]
\end{thm}
\begin{proof}
This is a consequence of the one-to-one correspondence between elements
of $\mathcal{C}'$ and $\widetilde{\mathcal{C}}'$ combined with the
bounded distortion property for $\varphi$ and $\psi$ (see Remark
\ref{rem:def_forhoelderpotentials}).
\end{proof}
Let us discuss the basic example, which motivates the constructions.
\begin{example}
[Gurevi\v c pressure] For an irreducible subshift of finite type
$\Sigma$ we consider the set $\mathcal{C}:=\Sigma_{a}^{\mathrm{per}}$
for a fixed element $a\in I$. Clearly, $\Sigma_{a}^{\mathrm{per}}$
is closed under concatenations and has the refinement property. The
$\Sigma_{a}^{\mathrm{per}}$--loop space is the full shift over the
alphabet $\Sigma_{a}^{\mathrm{smpl}}$, which we denote by $\widetilde{\Sigma}_{a}$.
Let $\varphi:\Sigma\rightarrow\R$ be a Hölder continuous potential.
As a first application of the invariance under inducing we express
the Gurevi\v c pressure via the classical pressure on the corresponding
loop space. By Theorem \ref{thm:invariance under inducing} and Corollary
\ref{cor:pressureformula_ifstreactlydescreasing} we have \begin{eqnarray*}
\mathcal{P}_{1}\left(\varphi,\Sigma_{a}^{\mathrm{per}}\right) & = & \mathcal{P}_{\tilde{1}}\left(\tilde{\varphi},\widetilde{\Sigma}_{a}^{*}\right)\\
 & = & \inf\left\{ \beta\in\R:\mathcal{P}_{1}\left(\tilde{\varphi}-\beta\tilde{1},\widetilde{\Sigma}_{a}^{*}\right)\le0\right\} \\
 & = & \sup\left\{ \beta\in\R:\mathcal{P}_{1}\left(\tilde{\varphi}-\beta\tilde{1},\widetilde{\Sigma}_{a}^{*}\right)\ge0\right\} .\end{eqnarray*}

Next, we consider $\mathcal{P}_{1}\left(\tilde{\varphi},\widetilde{\Sigma}_{a}^{*}\right)$,
which is equal to the classical pressure of $\tilde{\varphi}$ with
respect to the induced dynamical system $\left(\iota\left(\widetilde{\Sigma}_{a}\right),\sigma^{*}\right)$
with $\sigma^{*}=\sigma^{T}$, where $T\left(\omega\right):=\inf\left\{ n\ge1:\sigma^{n}\left(\omega\right)\in\left[a\right]\right\} $.
Since for $\psi:=\1_{\left[a\right]}$ on $\Sigma$ we have $\tilde{\psi}=1$
on $\widetilde{\Sigma}_{a}$ the invariance under inducing implies
\[
\mathcal{P}_{1}\left(\tilde{\varphi},\widetilde{\Sigma}_{a}^{*}\right)=\mathcal{P}_{\1_{\left[a\right]}}\left(\varphi,\Sigma_{a}^{\mathrm{per}}\right),\]
which motivates the term  $\psi$--induced pressure.
\end{example}
In the following example we consider the geometric potential of the
harmonic $\alpha$-Farey map introduced in \citep{MundayKessStrat10}.
We make use of the invariance under inducing from Theorem \ref{thm:invariance under inducing}
to determine the pressure. 
\begin{example}
[$\alpha$-Farey Map] \label{exa:lueroth-1} As an example let us
study the harmonic $\alpha$-Farey map $F$ defined in \citep{MundayKessStrat10},
since for this map all the important quantities can be calculated
explicitly. The map $F:\left[0,1\right]\rightarrow\left[0,1\right]$
is uniquely determined by the property of being linear on $I_{n}:=\left(1/\left(n+1\right),1/n\right],\: n\in\N$,
 such that \[
F\left(0\right)=0,\quad F\left(\frac{1}{n+1}\right)=\frac{1}{n},n\ge1,\quad F\left(1\right)=0.\]
 One immediately verifies that the geometric potential $\psi:=\log|F'|$
is positive but not bounded away from $0$. Furthermore, the family
of sets $\left(I_{n}\right)$ defines a Markov partition and with
respect to this partition $F$ is conjugated to the renewal shift
with state space $\N$ and root $1$, which we denote by $\Sigma$.
Notice, that $\Sigma$ is mixing but not finitely irreducible. 

We prove that $\mathcal{P}_{1}\left(-\beta\psi,\Sigma_{1}^{\mathrm{per}}\right)$
is strictly greater than zero for $\beta<1$ and equal to zero for
$\beta\ge1$. Furthermore, $\mathcal{P}_{1}\left(0,\Sigma_{1}^{\mathrm{per}}\right)=\log2$.
We also show that $P_{\psi}\left(0,\Sigma_{1}^{\mathrm{per}}\right)=1$,
which implies that $P_{\psi}\left(0,\Sigma_{1}^{\mathrm{per}}\right)$
is given by the pseudo inverse of the $1$--induced pressure. Furthermore,
$\sup\left\{ \beta\in\R:\mathcal{P}_{1}\left(-\beta\psi,\Sigma_{1}^{\mathrm{per}}\right)\ge0\right\} =\infty$. 

In order to determine \foreignlanguage{ngerman}{$\mathcal{P}_{1}\left(-\beta\psi,\Sigma_{1}^{\mathrm{per}}\right)$
for $\beta\in\R$ we consider the $\Sigma_{1}^{\mathrm{per}}$-loop
space denoted by $\widetilde{\Sigma}_{1}$. We have $\widetilde{\Sigma}_{1}=\left(\Sigma_{1}^{\mathrm{smpl}}\right)^{\N}$
and we abbreviate }$\left(1,k,k-1,\dots,2\right)\in\Sigma_{1}^{\mathrm{smpl}}$
by $\tilde{k}\in\N$. Observe that $\psi_{|\left[k\right]}=\log\left(\frac{k+1}{k-1}\right)$
for $k\ge2$ and $\psi_{|\left[1\right]}=\log2$ and hence, $\tilde{\psi}_{|\left[\tilde{k}\right]}=\log\left(k\left(k+1\right)\right)$
for $k\in\N$.

By the invariance under inducing (Theorem \ref{thm:invariance under inducing})
we have for $\beta\in\R$ \[
\mathcal{P}_{1}\left(-\beta\psi,\Sigma_{1}^{\mathrm{per}}\right)=\mathcal{P}_{\tilde{1}}\left(-\beta\tilde{\psi},\widetilde{\Sigma}_{1}^{*}\right)=\inf\left\{ t\in\R:\mathcal{P}_{1}\left(-\beta\tilde{\psi}-t\tilde{1},\widetilde{\Sigma}_{1}^{*}\right)\le0\right\} ,\]
 where $\tilde{1}_{|\left[\tilde{k}\right]}=k$ for $k\in\N$. We
have \[
\mathcal{P}_{1}\left(-\beta\tilde{\psi}-t\tilde{1},\widetilde{\Sigma}_{1}^{*}\right)=\log\sum_{k\in\N}\left(k\left(k+1\right)\right)^{-\beta}e^{-tk}.\]
Clearly, for every $\beta\le1$ there exists a unique $t\left(\beta\right)$
such that \[
\mathcal{P}_{1}\left(-\beta\tilde{\psi}-t\left(\beta\right)\tilde{1},\widetilde{\Sigma}_{1}^{*}\right)=0.\]
In particular, we have $t\left(0\right)=\log2$, $t\left(1\right)=0$
and $t\left(\beta\right)>0$ for $\beta<1$. For $\beta>1$ we have
that $\mathcal{P}_{1}\left(-\beta\tilde{\psi}-t\tilde{1},\widetilde{\Sigma}_{1}^{*}\right)$
is strictly less than zero for $t=0$ and is equal to infinity for
$t<0$, hence $\inf\left\{ t\in\R:\mathcal{P}_{1}\left(-\beta\tilde{\psi}-t\tilde{1},\widetilde{\Sigma}_{1}^{*}\right)\le0\right\} =0$
for every $\beta>1$. Furthermore, we have \[
P_{\psi}\left(0,\Sigma_{1}^{\mathrm{per}}\right)=P_{\tilde{\psi}}\left(0,\widetilde{\Sigma}_{1}^{*}\right)=1,\]
since by the above calculation $\mathcal{P}_{1}\left(-\tilde{\psi},\widetilde{\Sigma}_{1}^{*}\right)=0$. 
\end{example}

\subsection{Exhausting principles and group extensions}

In order to relate the induced pressure with respect to different
collections $\mathcal{C}\subset\Sigma^{*}$ it will important to introduce
mixing properties also with respect to such collections. 
\begin{defn}
\label{def:Finitely_irreducible_C} We say that $\mathcal{C}$ is
\emph{irreducible}, if for $C_{1},C_{2}\in\mathcal{C}$ there exists
$D\in\Sigma^{*}\cup\left\{ \emptyset\right\} $ such that $C_{1}DC_{2}\in\mathcal{C}$.
We say that $\mathcal{C}$ is \emph{finitely irreducible}, if there
exists a finite set $\Lambda\subset\Sigma^{*}$, such that for $C_{1},C_{2}\in\mathcal{C}$
there exists $D\in\Lambda\cup\left\{ \emptyset\right\} $ satisfying
$C_{1}DC_{2}\in\mathcal{C}$. \end{defn}
\begin{rem}
For $\mathcal{C}=\Sigma^{*}$ the definition coincides with the usual
definition of (finite) irreducibility of the incidence matrix of $\Sigma$.
\end{rem}
We also need the following notions of embeddability.
\begin{defn}
Let $\mathcal{C}',\mathcal{C}\subset\Sigma^{*}$. We say that\emph{
$\mathcal{C}$ is finitely embeddable into $\mathcal{C}'$}, if there
exists a finite set $\Lambda\subset\Sigma^{*}$ such that \[
\forall\omega\in\mathcal{C}\,\exists\tau_{1},\tau_{2}\in\Lambda:\,\tau_{1}\omega\tau_{2}\in\mathcal{C}'.\]
We say that $\mathcal{C}$ is\emph{ compactly finitely embeddable
into $\mathcal{C}'$}, if $\mathcal{C}\cap K^{*}$ is finitely embeddable
into $\mathcal{C}'$ for all compact $\sigma$-invariant subsets $K\subset\Sigma$. \end{defn}
\begin{lem}
\label{lem:pressure-finitelyembeddable}Let $\mathcal{C}',\mathcal{C}\subset\Sigma^{*}$
and $\varphi,\psi:\Sigma\rightarrow\R$ be Hölder continuous, $\psi\ge0$.
If $\mathcal{C}$ is finitely embeddable into $\mathcal{C}'$ then
\[
\mathcal{P}_{\psi}\left(\varphi,\mathcal{C}\right)\le\mathcal{P}_{\psi}\left(\varphi,\mathcal{C}'\right).\]
\end{lem}
\begin{proof}
For the proof it is sufficient to find a constant $C>0$, such that
for all $T>0$ large and $\beta\in\R$ we have \[
\sum_{{\omega\in\mathcal{C}\atop T<S_{\omega}\psi}}\exp S_{\omega}\left(\varphi-\beta\psi\right)\le C\sum_{{\omega\in\mathcal{C}'\atop T<S_{\omega}\psi}}\exp S_{\omega}\left(\varphi-\beta\psi\right).\]
Then the claim would follow by Theorem \ref{thm:pressure_as_inf}. 

In fact, let $\Lambda\subset\Sigma^{*}$ be the finite set witnessing
the finitely embeddability condition. Let $k$ denote the maximal
word length of the elements in $\Lambda$. Since $\varphi$ and $\psi$
are Hölder continuous we may choose $S_{\omega}\left(\varphi-\beta\psi\right):=\inf_{\tau\in\left[\omega\right]}S_{\left|\omega\right|}\left(\varphi-\beta\psi\right)\left(\tau\right)$
in the definition of the induced pressure. With $m:=\min\limits _{\tau\in\Lambda}S_{\tau}\left(\varphi-\beta\psi\right)$
we have\begin{eqnarray*}
\e^{2m}\sum_{{\omega\in\mathcal{C}\atop T<S_{\omega}\psi}}\exp S_{\omega}\left(\varphi-\beta\psi\right) & \le & \sum_{{\omega\in\mathcal{C}\atop T<S_{\omega}\psi}}\exp\left(S_{\tau_{1}\left(\omega\right)\omega\tau_{2}\left(\omega\right)}\left(\varphi-\beta\psi\right)\right).\end{eqnarray*}
Using this estimate and the fact that the map $\omega\mapsto\tau_{1}\left(\omega\right)\omega\tau_{2}\left(\omega\right)$
is at most $\left(2k\right)$-to-$1$ we finally conclude\[
\sum_{{\omega\in\mathcal{C}\atop T<S_{\omega}\psi}}\exp S_{\omega}\left(\varphi-\beta\psi\right)\le2k\e^{-2m}\sum_{{\omega\in\mathcal{C}'\atop T<S_{\omega}\psi}}\exp\left(S_{\tau_{1}\left(\omega\right)\omega\tau_{2}\left(\omega\right)}\left(\varphi-\beta\psi\right)\right).\]
\end{proof}
\begin{rem}
If in Lemma \ref{lem:pressure-finitelyembeddable} we have additionally
$\mathcal{C}'\subset\mathcal{C}$ then \[
\mathcal{P}_{\psi}\left(\varphi,\mathcal{C}\right)=\mathcal{P}_{\psi}\left(\varphi,\mathcal{C}'\right).\]

\end{rem}
Subsequently, the following notation will be useful. For an arbitrary
Markov shift $\Sigma$ we let $\pi_{j}:\Sigma\rightarrow I$ denote
the \emph{projection on the $j$-th coordinate}, i.e. $\pi_{j}\left(\omega\right)=\omega_{j}$. 

The next Lemma shows that the exhausting principle for the induced
version on the associated loop space implies the exhausting principle
for the original system. 
\begin{lem}
\label{lem:Cloopexhausting-Cexhausting}Let $\varphi,\psi:\Sigma\rightarrow\R$
be Hölder continuous, $\psi\ge0$. Assume that $\mathcal{C}$ is representable
by loops. If $\left(\tilde{\psi},\tilde{\varphi},\widetilde{\mathcal{C}}\right)$
satisfies the exhausting principle then also $\left(\psi,\varphi,\mathcal{C}\right)$
satisfies the exhausting principle. \end{lem}
\begin{proof}
For arbitrary $\epsilon>0$ we find a compact $\widetilde{\sigma}$-invariant
subset $\widetilde{K}\subset\widetilde{\Sigma}_{\mathcal{C}}$, such
that \[
\mathcal{P}_{\tilde{\psi}}\left(\tilde{\varphi},\widetilde{\mathcal{C}}\right)-\epsilon\le\mathcal{P}_{\tilde{\psi},\widetilde{K}}\left(\tilde{\varphi},\widetilde{\mathcal{C}}\right)=\mathcal{P}_{\tilde{\psi}}\left(\tilde{\varphi},\widetilde{\mathcal{C}}\cap\widetilde{K}^{*}\right),\]
where the last equality follows by Remark \ref{rem:def_forhoelderpotentials}. 

Since $\widetilde{K}$ is compact we have that $\tilde{I}_{K}:=\pi_{1}(\widetilde{K})\subset\mathcal{C}^{\mathrm{smpl}}$
is compact and hence finite. The $\widetilde{\sigma}$-invariance
of $\widetilde{K}$ implies that we have $\pi_{n}(\widetilde{K})=\pi_{1}(\widetilde{\sigma}^{n}\widetilde{K})=\pi_{1}(\widetilde{K})\subset\tilde{I}_{K}$
and therefore $\widetilde{K}\subset\left(\tilde{I}_{K}\right)^{\N}$.
By Proposition \ref{fac:pressure} (\ref{enu:(Monotonicity)}) and
Theorem \ref{thm:invariance under inducing} we have \[
\mathcal{P}_{\tilde{\psi}}\left(\tilde{\varphi},\widetilde{\mathcal{C}}\cap\widetilde{K}^{*}\right)\le\mathcal{P}_{\tilde{\psi}}\left(\tilde{\varphi},\widetilde{\mathcal{C}}\cap\tilde{I}_{K}^{*}\right)=\mathcal{P}_{\psi}\left(\varphi,\mathcal{C}\cap\left\{ \omega\in\Sigma^{*}:\tilde{\omega}_{i}\in\tilde{I}_{K}\right\} \right).\]
Next, we define $I_{K}$ as the (finite) set of all elements in $I$,
which are needed to represent the elements of $\tilde{I}_{K}$. Then
$K:=\left(I_{K}\right)^{\N}\cap\Sigma$ is compact and $\sigma$-invariant.
By Proposition \ref{fac:pressure} (\ref{enu:(Monotonicity)}) and
Remark \ref{rem:def_forhoelderpotentials} \[
\mathcal{P}_{\psi}\left(\varphi,\mathcal{C}\cap\left\{ \omega\in\Sigma^{*}:\tilde{\omega}_{i}\in\tilde{I}_{K}\right\} \right)\le\mathcal{P}_{\psi}\left(\varphi,\mathcal{C}\cap K^{*}\right)=\mathcal{P}_{\psi,K}\left(\varphi,\mathcal{C}\right).\]
Finally combining Theorem \ref{thm:invariance under inducing} with
the above estimates we conclude that $\mathcal{P}_{\psi}\left(\varphi,\mathcal{C}\right)-\epsilon=\mathcal{P}_{\tilde{\psi}}\left(\tilde{\varphi},\widetilde{\mathcal{C}}\right)-\epsilon\le\mathcal{P}_{\psi,K}\left(\varphi,\mathcal{C}\right)$. \end{proof}
\begin{rem}
We would like to point out that the conditions of Lemma \ref{lem:Cloopexhausting-Cexhausting}
in particular imply that by Corollary \ref{cor:psipressure_asinf-underexhaustingprop}
we have \[
\mathcal{P}_{\psi}\left(\varphi,\mathcal{C}\right)=\inf\left\{ \beta\in\R:\mathcal{P}_{1}\left(\varphi-\beta\psi,\mathcal{C}\right)\le0\right\} ,\]
whereas in general $\mathcal{P}_{\psi}\left(\varphi,\mathcal{C}\right)$
is not equal to $\sup\left\{ \beta\in\R:\mathcal{P}_{1}\left(\varphi-\beta\psi,\mathcal{C}\right)\ge0\right\} $
as seen in Example \ref{exa:lueroth-1}.
\end{rem}
Next Theorem considers subsystems $\mathcal{C}'\subset\mathcal{C}$
with $\mathcal{C}$ compactly finitely embeddable into $\mathcal{C}'$.
Under the condition that $\mathcal{C}'$ is finitely irreducible we
prove that $\mathcal{P}_{\psi}\left(\varphi,\mathcal{C}\right)=\mathcal{P}_{\psi}\left(\varphi,\mathcal{C}'\right)$
is equivalent to the exhausting principle for $\left(\psi,\varphi,\mathcal{C}\right)$. 
\begin{thm}
\label{thm:exhaustable-part}Let $\varphi,\psi:\Sigma\to\mathbb{R}$
 be Hölder continuous with $\psi>0$. Let $\mathcal{C}'\subset\mathcal{C}\subset\Sigma^{*}$
and suppose that $\mathcal{C}'$ is finitely irreducible and representable
by loops, and that  $\mathcal{C}$ be compactly finitely embeddable
into $\mathcal{C}'$. Then \[
\sup_{K\in C_{\Sigma,\sigma}}\mathcal{P}_{\psi,K}\left(\varphi,\mathcal{C}\right)=\mathcal{P}_{\psi}\left(\varphi,\mathcal{C}'\right).\]
In particular, we have \[
\left(\psi,\varphi,\mathcal{C}\right)\,\mbox{satisfies the exhausting principle }\iff\mathcal{P}_{\psi}\left(\varphi,\mathcal{C}\right)=\mathcal{P}_{\psi}\left(\varphi,\mathcal{C}'\right).\]
\end{thm}
\begin{proof}
Since $\mathcal{C}'$ is finitely irreducible and $\tilde{\varphi},\tilde{\psi}$
are Hölder continuous, making a similar calculation as in the proof
of \citep[Theorem 2.1.5]{MR2003772} we readily find that $\left(1,\tilde{\varphi}-\beta\tilde{\psi},\widetilde{\Sigma}_{\mathcal{C}'}^{*}\right)$
satisfies the exhausting principle for each $\beta\in\R$. Next, we
show that also $\left(\tilde{\psi},\tilde{\varphi},\widetilde{\Sigma}_{\mathcal{C}'}^{*}\right)$
satisfies the exhausting principle. For this, let $\delta>0$. By
Corollary \ref{cor:pressureformula_ifstreactlydescreasing} we have
$\mathcal{P}_{\tilde{\psi}}\left(\tilde{\varphi},\widetilde{\Sigma}_{\mathcal{C}'}^{*}\right)=\inf\left\{ \beta\in\R:\mathcal{P}_{1}\left(\tilde{\varphi}-\beta\tilde{\psi},\widetilde{\Sigma}_{\mathcal{C}'}^{*}\right)\le0\right\} $
and hence \[
\mathcal{P}_{1}\left(\tilde{\varphi}-\left(\mathcal{P}_{\tilde{\psi}}\left(\tilde{\varphi},\widetilde{\Sigma}_{\mathcal{C}'}^{*}\right)-\delta\right)\tilde{\psi},\widetilde{\Sigma}_{\mathcal{C}'}^{*}\right)>0.\]
Since the exhausting principle holds for $\left(1,\tilde{\varphi}-\beta\tilde{\psi},\widetilde{\Sigma}_{\mathcal{C}'}^{*}\right)$
with $\beta\in\R$, we find a compact $\tilde{\sigma}$-invariant
subset $\widetilde{K}\subset\widetilde{\Sigma}_{\mathcal{C}'}$, such
that\[
\mathcal{P}_{1,\widetilde{K}}\left(\tilde{\varphi}-\left(\mathcal{P}_{\tilde{\psi}}\left(\tilde{\varphi},\widetilde{\Sigma}_{\mathcal{C}'}^{*}\right)-\delta\right)\tilde{\psi},\widetilde{\Sigma}_{\mathcal{C}'}^{*}\right)>0.\]
 Using Corollary \ref{cor:pressure_rescaling} we get \[
\mathcal{P}_{\tilde{\psi},K}\left(\tilde{\varphi},\mathcal{C}'\right)\ge\inf\left\{ \beta\in\R:\mathcal{P}_{1,\widetilde{K}}\left(\tilde{\varphi}-\beta\tilde{\psi},\mathcal{C}\right)\le0\right\} \ge\mathcal{P}_{\tilde{\psi}}\left(\tilde{\varphi},\mathcal{C}'\right)-\delta\]
and hence $\left(\tilde{\psi},\tilde{\varphi},\widetilde{\Sigma}_{\mathcal{C}'}^{*}\right)$
satisfies the exhausting principle.

By Lemma \ref{lem:Cloopexhausting-Cexhausting} also $\left(\psi,\varphi,\mathcal{C}'\right)$
satisfies the exhausting principle. Using this and Proposition \ref{fac:pressure}
(\ref{enu:(Monotonicity)}) we therefore have \[
\sup_{K\in C_{\Sigma,\sigma}}\mathcal{P}_{\psi,K}\left(\varphi,\mathcal{C}\right)\ge\sup_{K\in C_{\Sigma,\sigma}}\mathcal{P}_{\psi,K}\left(\varphi,\mathcal{C}'\right)=\mathcal{P}_{\psi}\left(\varphi,\mathcal{C}'\right).\]
For the reverse inequality notice that $\mathcal{C}\cap K^{*}$ is
finitely embeddable into $\mathcal{C}'$ for $K\in C_{\Sigma,\sigma}$,
which by Lemma \ref{lem:pressure-finitelyembeddable} implies that
\[
\mathcal{P}_{\psi,K}\left(\varphi,\mathcal{C}\right)\le\mathcal{P}_{\psi}\left(\varphi,\mathcal{C}'\right).\]
\end{proof}
\begin{cor}
\label{cor:exhausting-if-finitely-irred}For Hölder continuous functions
$\varphi,\psi$ with $\psi>0$ and $\mathcal{C}\subset\Sigma^{*}$
finitely irreducible and representable by loops, we have that $\left(\psi,\varphi,\mathcal{C}\right)$
satisfies the exhausting principle. \end{cor}
\begin{defn}
For a fixed starting set $J_{s}\subset I$ and a terminating set $J_{t}\subset I$
we denote by $\mathcal{C}=\left\{ \omega\in\Sigma^{*}:\omega_{1}\in J_{s}\right\} $
the \emph{set of starting words in $J_{s}$} and by $\mathcal{C}'=\left\{ \omega\in\mathcal{C}:\omega_{1}\in J_{s},\omega_{\left|\omega\right|}\in J_{t}\right\} $
the \emph{set of bridges from $J_{s}$ to $J_{t}$. }\end{defn}
\begin{cor}
[Exhausting principle for bridges]\label{cor:bridge-exhausting}Let
the incidence matrix of $\Sigma$ be irreducible and let $\varphi,\psi:\Sigma\rightarrow\R$
be Hölder continuous, $\psi>0$. Furthermore, let $\mathcal{C}$ be
the set of words starting in the non-empty and finite set $J_{s}\subset I$
and let $\mathcal{C}'$ be the set of bridges from $J_{s}$ to the
non-empty and finite set $J_{t}\subset I$. Then \[
\sup_{K\in C_{\Sigma,\sigma}}\mathcal{P}_{\psi,K}\left(\varphi,\mathcal{C}\right)=\mathcal{P}_{\psi}\left(\varphi,\mathcal{C}'\right).\]
In particular, we have \[
\left(\psi,\varphi,\mathcal{C}\right)\,\mbox{satisfies the exhausting principle }\iff\mathcal{P}_{\psi}\left(\varphi,\mathcal{C}\right)=\mathcal{P}_{\psi}\left(\varphi,\mathcal{C}'\right).\]
\end{cor}
\begin{proof}
We will verify the assumptions of Theorem \ref{thm:exhaustable-part}.
Clearly, $\mathcal{C}'$ is closed under concatenations, has the refinement
property and is hence representable by loops. Since $\Sigma$ is irreducible
and the sets $J_{s}$ and $J_{t}$ are finite we conclude that $\mathcal{C}'$
is finitely irreducible. It remains to show that $\mathcal{C}$ is
compactly finitely embeddable into $\mathcal{C}'$. Let $K\in C_{\Sigma,\sigma}$
be compact and $\sigma$-invariant. Repeating the arguments in the
proof of Lemma \ref{lem:Cloopexhausting-Cexhausting} we find that
$K\subset\Sigma_{N}$ for some $N\in\N$. Fix some $\gamma\in\mathcal{C}'$.
By the definition of $\mathcal{C}'$ and the fact that $\Sigma$ is
irreducible we find for all $\omega\in\mathcal{C}\cap K^{*}$ elements
$\tau_{1},\tau_{2}\in\Sigma^{*}$ such that $\gamma\tau_{1}\omega\tau_{2}\gamma\in\mathcal{C}'$.
In fact, $\tau_{1},\tau_{2}$ can be taken from a finite set $\Lambda_{K}\subset\Sigma^{*}$,
since the elements $\omega\in\mathcal{C}\cap K^{*}$ are constructed
over the finite alphabet $\left\{ 1,\ldots,N\right\} $. \end{proof}
\begin{example}
[Simple Random Walk]\label{exa:Z-extension}By Corollary \ref{cor:exhausting-if-finitely-irred}
we know that a sufficient condition for the exhausting principle to
hold for $\left(\psi,\varphi,\mathcal{C}\right)$ and arbitrary Hölder
continuous functions $\varphi$ and $\psi>0$ is that $\mathcal{C}$
is finitely irreducible and representable by loops. This example shows
that this condition is not necessary.

Let us consider \[
\Sigma_{\Z}=\left\{ \left(\tau_{n},h_{n}\right)\in\left(\left\{ -1,+1\right\} \times\Z\right)^{\N}:h_{n}+\tau_{n}=h_{n+1}\right\} \]
together with the starting set $J_{s}:=\left\{ \left(1,0\right),\left(-1,0\right)\right\} $
and terminating set $J_{t}:=\left\{ \left(1,-1\right),\left(-1,1\right)\right\} $.
The corresponding sets of starting words in $J_{s}$ and bridges from
$J_{s}$ to $J_{t}$ are then given by \[
\mathcal{C}:=\left\{ \omega\in\Sigma_{\Z}^{*}:h_{1}=0\right\} ,\quad\mathcal{C}':=\left\{ \omega\in\mathcal{C}:h_{\left|\omega\right|}+\tau_{\left|\omega\right|}=0\right\} .\]
Notice that $\Sigma_{\Z}$ is irreducible. It is also evident that
$\mathcal{C}$ is not finitely irreducible. Nevertheless, we will
show that $\left(1,0,\mathcal{C}\right)$ satisfies the exhausting
principle. By Corollary \ref{cor:bridge-exhausting} this is equivalent
to $\mathcal{P}_{1}\left(0,\mathcal{C}\right)=\mathcal{P}_{1}\left(0,\mathcal{C}'\right)$. 

Clearly, $\mathcal{P}_{1}\left(0,\mathcal{C}\right)=\log2$ and we
will prove that also $\mathcal{P}_{1}\left(0,\mathcal{C}'\right)=\log2$.
For this note that $\card\left\{ \omega\in\mathcal{C}',\left|\omega\right|=2n\right\} ={2n \choose n}$,
which is by Stirling's formula comparable to $2^{2n}n^{-1/2}$. Taking
logarithm and dividing by $2n$ proves the assertion.

Let us also give an alternative proof using the invariance under inducing
as stated in Theorem \ref{thm:invariance under inducing}. We have
\[
\mathcal{P}_{1}\left(0,\mathcal{C}'\right)=\mathcal{P}_{\tilde{1}}\left(0,\widetilde{\Sigma}_{\mathcal{C}'}^{*}\right)=\inf\left\{ t\in\R:\mathcal{P}_{1}\left(-\beta\tilde{1},\widetilde{\Sigma}_{\mathcal{C}'}^{*}\right)\le0\right\} .\]
 Since $\widetilde{\Sigma}_{\mathcal{C}'}$ is a full shift, \[
\mathcal{P}_{1}\left(-\beta\tilde{1},\widetilde{\Sigma}_{\mathcal{C}'}^{*}\right)=\log\sum_{\omega\in\left(\mathcal{C}'\right)^{\mathrm{smpl}}}\e^{-\beta\left|\omega\right|}.\]
Following \citep{MR0109367} we compute \[
\card\left\{ \omega\in\left(\mathcal{C}'\right)^{\mathrm{smpl}},\left|\omega\right|=2n\right\} =\frac{2}{n}{2n-2 \choose n-1}\]
and using the Binomial Theorem we find for $\beta\ge\log2$ that $\mathcal{P}_{1}\left(-\beta\tilde{1},\widetilde{\Sigma}_{\mathcal{C}'}^{*}\right)$
is finite and equal to $\log\big(\,1-\sqrt{1-\left(\log2-\beta\right)^{2}}\,\big)$.
We conclude $\mathcal{P}_{1}\left(-\log2\cdot\tilde{1},\widetilde{\Sigma}_{\mathcal{C}'}^{*}\right)=0$
and hence $\mathcal{P}_{1}\left(0,\mathcal{C}'\right)=\log2$. 
\end{example}
In fact, the above example is a special case of the following general
Theorem. Let $\mathbb{F}_{k}:=\left\langle g_{1},\dots,g_{k}\right\rangle $
denote the free group of rank $k\ge1$ and define $I:=\left\{ g_{1},\dots,g_{k},g_{1}^{-1},\dots,g_{k}^{-1}\right\} $
to be the set of symmetric generators of $\mathbb{F}_{k}$. Let $N$
be a normal subgroup of $\mathbb{F}_{k}$ and let $\pi:\mathbb{F}_{k}\rightarrow\mathbb{F}_{k}/N=:G$
be the canonical factor map. Let us consider the two naturally associated
Markov shifts \[
\Sigma_{G}:=\left\{ \left(\tau_{n},h_{n}\right)\in\left(I\times G\right)^{\N}:h_{n}\pi\left(\tau_{n}\right)=h_{n+1}\right\} \]
and \[
\underline{\Sigma}_{G}:=\left\{ \left(\tau_{n},h_{n}\right)\in\left(I\times G\right)^{\N}:\tau_{n}\tau_{n+1}\neq\id,h_{n}\pi\left(\tau_{n}\right)=h_{n+1}\right\} ,\]
where in the latter case we always assume $k\ge2$ and $N\neq\left\{ \id\right\} $. 

Note, that this construction was introduced before e.g.\ in \citep{MR2337557},
where it was called the modular shift space in the context of finite
index subgroups of $\mbox{PSL}_{2}\left(\mathbb{Z}\right)$. This
representation is closely connected to a skew product dynamical system,
since for $\left(\tau_{n},h_{n}\right)\in\Sigma_{G}$ or $\underline{\Sigma}_{G}$
we have $h_{n}=h_{1}\pi\left(\tau_{1}\right)\cdots\pi\left(\tau_{n}\right)$
for $n\ge2$. 
\begin{thm}
[Group extension] \label{thm:-Group-Extension}For $\Xi$ either
equal to $\Sigma_{G}$ or $\underline{\Sigma}_{G}$, let $\mathcal{C}'\subset\mathcal{C}$
be given by \[
\mathcal{C}:=\left\{ \omega\in\Xi:h_{1}=\id\right\} \;\textrm{ and }\;\mathcal{C}':=\left\{ \omega\in\mathcal{C}:h_{\left|\omega\right|}\pi\left(\tau_{\left|\omega\right|}\right)=\id\right\} .\]
We have the following chain of equivalences. \begin{eqnarray*}
G\,\mbox{is amenable} & \iff & \left(1,0,\mathcal{C}\right)\,\mbox{satisfies the exhausting principle}\\
 & \iff & \mathcal{P}_{1}\left(0,\mathcal{C}\right)=\mathcal{P}_{1}\left(0,\mathcal{C}'\right).\end{eqnarray*}
Moreover,\[
G\;\mbox{ finite }\iff\mathcal{C}\;\mbox{ finite irreducible }\iff\;\Xi\;\mbox{compact.}\]
 \end{thm}
\begin{rem}
\label{rem:correspondence-C-Fk}For $\Xi=\underline{\Sigma}_{G}$
there is a canonical one-to-one correspondence between both the sets
$\mathcal{C}$ and $\mathbb{F}_{k}$, and the sets $\mathcal{C}'$
and $N$. This allows an interpretation of our result in terms of
the Poincaré exponent. \end{rem}
\begin{proof}
One easily verifies that the incidence matrix of $\Xi$ is irreducible.
Note that $\mathcal{C}$ is in fact the set of starting words in $J_{s}:=\left\{ \left(g,\id\right):g\in I\right\} $
and $\mathcal{C}'$ is the set of bridges from $J_{s}$ to $J_{t}:=\left\{ \left(g,\pi\left(g^{-1}\right)\right):g\in I\right\} $.
Hence, by Corollary \ref{cor:bridge-exhausting} we have \[
\left(1,0,\mathcal{C}\right)\,\mbox{satisfies the exhausting principle\ensuremath{\iff}}\mathcal{P}_{1}\left(0,\mathcal{C}\right)=\mathcal{P}_{1}\left(0,\mathcal{C}'\right).\]
Let us first consider the case $\Xi=\Sigma_{G}$. In order to prove
that this is also equivalent to $G$ being amenable we make use of
Kesten's characterisation of finitely generated amenable groups (\citep{MR0109367,MR0112053}).
We introduce the stochastic matrix \[
p\left(g,g'\right):=\frac{1}{2k}\card\left\{ 1\le i\le2k:g\pi\left(g_{i}\right)=g'\right\} ,\quad g,g'\in G,\]
which defines a symmetric random walk on $G$ with the property that
\[
\mathcal{P}_{1}\left(0,\mathcal{C}'\right)=\log\left(2k\right)+\limsup_{n}\frac{1}{n}\log p^{\left(n\right)}\left(\id,\id\right).\]
By (\citep[Main Theorem]{MR0112053}) we have $\limsup_{n}\frac{1}{n}\log p^{\left(n\right)}\left(\id,\id\right)=0$,
if and only if $G$ is amenable. Since $\mathcal{P}_{1}\left(0,\mathcal{C}\right)=\log\left(2k\right)$
the assertion follows.

Next we consider the case $\Xi=\underline{\Sigma}_{G}$ for $k\ge2$
and $N\neq\left\{ \id\right\} $. Clearly, in this case we have $\mathcal{P}_{1}\left(0,\mathcal{C}\right)=\log\left(2k-1\right)$.
By Remark \ref{rem:correspondence-C-Fk} the elements of $\mathcal{C}'$
correspond to the elements of $N$. By Grigorchuk's cogrowth criterion
(see \citep{MR599539} and also \citep{MR678175,MR0214137,MR1929333})
we conclude that $\mathcal{P}_{1}\left(0,\mathcal{C}'\right)=\log\left(2k-1\right)$,
if and only if $G$ is amenable. This finishes the proof of the first
part. 

The second statement of the theorem is in both cases an immediate
consequence of the construction of $\Xi$ and $\mathcal{C}$. In fact,
for $G$ finite, $\Xi$ is a subshift of finite type with finite state
space. 
\end{proof}
Having regard to the results of Brooks in \citep{MR783536} we expect
the above characterisation of amenability to hold true for a much
larger class of Hölder continuous functions $\varphi$ and $\psi>0$.

\subsection{Gurevi\v c pressure}

In this section we apply our results to the Gurevi\v c pressure (cf.\
\citep{MR0263162,MR0268356,MR1738951,MR1818392,MR1955261}). Recall
that the\emph{ }Gurevi\v c\emph{ }pressure is defined only for mixing
subshifts of finite type and in that setting coincides with $\mathcal{P}_{1}\left(\varphi,\Sigma_{a}^{\mathrm{per}}\right)$
for $a\in I$. 

As a direct application of the invariance under inducing of Theorem
\ref{thm:invariance under inducing}, Lemma \ref{lem:Cloopexhausting-Cexhausting},
and Corollary \ref{cor:psipressure_asinf-underexhaustingprop} we
give a new description of the Gurevi\v c pressure in terms of the
classical pressure for infinite systems as investigated e.g.\ by
Mauldin and Urba\'nski in \citep{MR2003772}.
\begin{cor}
\label{cor:ExhaustGurevich} Let $\varphi,\psi:\Sigma\rightarrow\R$
be Hölder continuous, $\psi>0$. Then the exhausting principle holds
for $\left(\psi,\varphi,\Sigma_{a}^{\mathrm{per}}\right)$ and any
$a\in I$. Furthermore, we have \begin{eqnarray*}
\mathcal{P}_{\psi}\left(\varphi,\Sigma_{a}^{\mathrm{per}}\right) & = & \inf\left\{ \beta:\mathcal{P}_{1}\left(\varphi-\beta\psi,\Sigma_{a}^{\mathrm{per}}\right)\le0\right\} \\
 & = & \inf\left\{ \beta\in\R:\mathcal{P}_{1}\left(\tilde{\varphi}-\beta\tilde{\psi},\widetilde{\Sigma}_{a}^{\mathrm{per}}\right)\le0\right\} \\
 & = & \sup\left\{ \beta\in\R:\mathcal{P}_{1}\left(\tilde{\varphi}-\beta\tilde{\psi},\widetilde{\Sigma}_{a}^{\mathrm{per}}\right)\ge0\right\} .\end{eqnarray*}
\end{cor}
\begin{rem}
In particular, for potentials $\psi$, such that the induced version
$\tilde{\psi}$ on the associated $\Sigma_{a}^{\mathrm{per}}$-loop
space is constant on cylindrical sets of words of lenght one \[
\mathcal{P}_{\psi}\left(0,\Sigma_{a}^{\mathrm{per}}\right)=\inf\left\{ \beta\in\R:\sum_{\omega\in\Sigma_{a}^{\mathrm{smpl}}}e^{-\beta S_{\omega}\psi}<1\right\} .\]
See also the dicussion of the special semi-flow presented in the introduction. 
\end{rem}
It is well known that the Gurevi\v c pressure for $\psi=1$ and a
mixing subshift of finite type is independent of $a\in I$. The following
generalisation for $\psi$--induced pressure follows from Corollary
\ref{cor:ExhaustGurevich} and Lemma \ref{lem:pressure-finitelyembeddable},
since $\Sigma_{a}^{\mathrm{per}}$ is finitely embeddable into $\Sigma_{b}^{\mathrm{per}}$
for $a,b\in I$. 
\begin{fact}
Let $\varphi,\psi:\Sigma\rightarrow\R$ be Hölder continuous, $\psi>0$,
and $A$ irreducible. Then we have that $\mathcal{P}_{\psi}\left(\varphi,\Sigma_{a}^{\mathrm{per}}\right)$
is independent of the choice of $a\in I$. 
\end{fact}
As a consequence of Proposition \ref{fac:pressure} (\ref{enu:(Monotonicity)})
we have for the particular choices $\mathcal{C}=\Sigma_{a}^{\mathrm{per}}$,
$\mathcal{C}=\Sigma^{\mathrm{per}}$ and $\mathcal{C}=\Sigma^{*}$
the following relation. 
\begin{fact}
For $\varphi,\psi:\Sigma\rightarrow\R$ Hölder continuous, $\psi\ge0$,
we have \[
\mathcal{P}_{\psi}\left(\varphi,\Sigma_{a}^{\mathrm{per}}\right)\le\mathcal{P}_{\psi}\left(\varphi,\Sigma^{\mathrm{per}}\right)\le\mathcal{P}_{\psi}\left(\varphi,\Sigma^{*}\right).\]
 
\end{fact}
For the following, we denote the \emph{irreducible component of $\Sigma$
containing $a$} by $\Sigma\left(a\right):=\left\{ \omega\in\Sigma:a\curvearrowright\omega_{1},\forall i\in\N\:\omega_{i}\curvearrowright a\right\} $,
where $a\curvearrowright b$ means, that there exists $\tau\in\Sigma^{*}$
with $a\tau b\text{\ensuremath{\in\Sigma}}^{*}$. Theorem \ref{thm:exhaustable-part}
provides us with a dichotomy for the Gurevi\v c pressure and the
classical pressure.
\begin{cor}
[Classicle Pressure--Gurevi{\v c} Pressure Dichotomy]\label{cor:urbgur-criterion}
Let $\varphi,\psi:\Sigma\rightarrow\R$ be Hölder continuous with
$\psi>0$. Then the following holds.
\begin{enumerate}
\item \textup{$\sup_{K\in C_{\Sigma,\sigma}}\mathcal{P}_{\psi,K}\left(\varphi,\Sigma\left(a\right)^{*}\right)=\mathcal{P}_{\psi}\left(\varphi,\Sigma_{a}^{\mathrm{per}}\right)$
for $a\in I$. }
\item $\sup_{K\in C_{\Sigma,\sigma}}\mathcal{P}_{\psi,K}\left(\varphi,\Sigma^{*}\right)=\sup_{a\in I}\mathcal{P}_{\psi}\left(\varphi,\Sigma_{a}^{\mathrm{per}}\right).$
\end{enumerate}
In particular, if $\Sigma$ is irreducible then we have for every
$a\in I$, \[
\sup_{K\in C_{\Sigma,\sigma}}\mathcal{P}_{\psi,K}\left(\varphi,\Sigma^{*}\right)=\mathcal{P}_{\psi}\left(\varphi,\Sigma_{a}^{\mathrm{per}}\right)\]
and the following equivalence holds: \[
\left(\psi,\varphi,\Sigma^{*}\right)\,\mbox{satisfies the exhausting principle }\iff\mathcal{P}_{\psi}\left(\varphi,\Sigma^{*}\right)=\mathcal{P}_{\psi}\left(\varphi,\Sigma_{a}^{\mathrm{per}}\right).\]
\end{cor}
\begin{proof}
ad (1): The {}``$\ge$''--part follows, since $\Sigma\left(a\right)^{*}\supset\Sigma_{a}^{\mathrm{per}}$
and $\left(\psi,\varphi,\Sigma_{a}^{\mathrm{per}}\right)$ satisfies
the exhausting principle by Corollary \ref{cor:ExhaustGurevich}.
For the {}``$\le$''--part recall that every $K\in C_{\Sigma,\sigma}$
is contained in a set $\Sigma_{N}$ for some $N\in\N$. By definition
of $\Sigma\left(a\right)$ we find for every element $\omega\in\Sigma\left(a\right)^{*}$
elements $\tau_{1},\tau_{2}\in\Sigma^{*}$ such that $\tau_{1}\omega\tau_{2}\in\Sigma_{a}^{\mathrm{per}}$.
For $\omega\in\Sigma\left(a\right)^{*}\cap\Sigma_{N}$ the elements
$\tau_{1},\tau_{2}$ can be chosen from a finite set. We conclude
that $\Sigma\left(a\right)^{*}\cap\Sigma_{N}$ is finitely embeddable
into $\Sigma_{a}^{\mathrm{per}}$ and hence the upper bound follows
by Lemma \ref{lem:pressure-finitelyembeddable}. 

ad (2): We only comment on the {}``$\le$''--part. Let $K\in C_{\Sigma,\sigma}$
with $K\subset\Sigma_{N}$ for some $N\in\N$. We have \[
\mathcal{P}_{\psi,K}\left(\varphi,\Sigma^{*}\right)\le\max_{n\in\left\{ 1,\dots,N\right\} }\mathcal{P}_{\psi,\Sigma_{N}}\left(\varphi,\Sigma\left(n\right)^{*}\right)\le\max_{n\in\left\{ 1,\dots,N\right\} }\mathcal{P}_{\psi,\Sigma_{N}}\left(\varphi,\Sigma_{n}^{\mathrm{per}}\right),\]
where the first inequality follows by decomposition in the finitely
many irreducible components and the second by an application of part
(1).
\end{proof}
Restricting to finitely irreducible systems we obtain 
\begin{cor}
Let the incidence matrix of $\Sigma$ be finitely irreducible and
let $\varphi,\psi:\Sigma\rightarrow\R$ be Hölder continuous, $\psi>0$.
Then we have\textup{\[
\mathcal{P}_{\psi}\left(\varphi,\Sigma_{a}^{\mathrm{per}}\right)=\mathcal{P}_{\psi}\left(\varphi,\Sigma^{\mathrm{per}}\right)=\mathcal{P}_{\psi}\left(\varphi,\Sigma^{*}\right)\]
and\[
\mathcal{P}_{\psi}\left(\varphi,\Sigma_{a}^{\mathrm{per}}\right)=\inf\left\{ \beta:\mathcal{P}_{1}\left(\varphi-\beta\psi,\Sigma^{*}\right)\le0\right\} =\sup\left\{ \beta:\mathcal{P}_{1}\left(\varphi-\beta\psi,\Sigma^{*}\right)\ge0\right\} .\]
}
\end{cor}

Next, we give an example where $\mathcal{P}_{\psi}\left(0,\Sigma_{1}^{\mathrm{per}}\right)<\mathcal{P}_{\psi}\left(0,\Sigma^{*}\right)$
whereas $\mathcal{P}_{\psi}\left(\varphi,\Sigma_{1}^{\mathrm{per}}\right)=\mathcal{P}_{\psi}\left(\varphi,\Sigma^{*}\right)$
for some Hölder continuous potentials $\psi,\varphi:\Sigma\rightarrow\R$
with $\psi>0$. Necessarily, the incidence matrix in this example
is not finitely irreducible. 
\begin{example}
\label{exa:renewalshift}We consider again the renewal shift $\Sigma$
and the potential $\psi:\Sigma\rightarrow\R$ satisfying $\psi_{|\left[n\right]}=\log\left(\frac{n+1}{n-1}\right)$
for $n\ge2$ and $\psi_{|\left[1\right]}=\log2$ from Example \ref{exa:lueroth-1}.
In Example \ref{exa:lueroth-1} we proved $\mathcal{P}_{\psi}\left(0,\Sigma_{1}^{\mathrm{per}}\right)=1$,
now we consider $\mathcal{P}_{\psi}\left(0,\Sigma^{*}\right)$. We
calculate that for $n,k\in\N$ with $1\le k<n$ and \foreignlanguage{ngerman}{$\tau\in\left[n,n-1,\dots,n-k+1\right]$}
\[
S_{k}\psi\left(\tau\right)=\log\left(\left(n\left(n+1\right)\right)/\left(\left(n-k+1\right)\left(n-k\right)\right)\right).\]
We conclude that for all $T>0$ and $\beta\in\R$ we have $\sum_{{\omega\in\Sigma^{*}\atop T<S_{\omega}\psi}}\e^{-\beta S_{\omega}\psi}=\infty$.
Hence, by Theorem \ref{thm:pressure_as_inf} we have $\mathcal{P}_{\psi}\left(0,\Sigma^{*}\right)=\infty$.
Since $\Sigma^{*}$ is compactly finitely embeddable into $\Sigma_{1}^{\mathrm{per}}$
we infer that $\left(\psi,0,\Sigma^{*}\right)$ does not satisfy the
exhausting principle by Theorem \ref{thm:exhaustable-part}.

Next, we introduce the potential $\varphi:\Sigma\rightarrow\R$ given
by $\varphi\left(\omega\right):=-\omega_{1}$. We verify that $\mathcal{P}_{\psi}\left(\varphi,\Sigma_{1}^{\mathrm{per}}\right)=\mathcal{P}_{\psi}\left(\varphi,\Sigma^{*}\right)$
and hence $\left(\psi,\varphi,\Sigma^{*}\right)$ satisfies the exhausting
principle. By decomposing the partition function corresponding to
$\Sigma^{*}$ in a product we obtain\begin{eqnarray*}
\sum_{\omega\in\Sigma^{*}}\e^{S_{\omega}\left(\varphi-\beta\psi\right)} & = & \left(1+\sum_{n\ge2}\e^{S_{\left(n,n-1,\dots,2\right)}\left(\varphi-\beta\psi\right)}\right)\sum_{\omega\in\Sigma_{1}^{\mathrm{per}}}\e^{S_{\omega}\left(\varphi-\beta\psi\right)}\\
 &  & \qquad\qquad\times\left(1+\sum_{N\geq2}\sum_{1\le k<N}\e^{S_{\left(1,N,N-1,\dots,N-k+1\right)}\left(\varphi-\beta\psi\right)}\right).\end{eqnarray*}
Using Theorem \ref{thm:pressure_as_inf} we conclude that $\mathcal{P}_{\psi}\left(\varphi,\Sigma^{*}\right)$
is the maximum of $\mathcal{P}_{\psi}\left(\varphi,\Sigma_{1}^{\mathrm{per}}\right)$,
$\mathcal{P}_{\psi}\left(\varphi,\left\{ \left(n,n-1,\ldots,1\right):n\ge2\right\} \right)$,
and \[
\mathcal{P}_{\psi}\left(\varphi,\left\{ \left(1,N,\ldots,N-k+1\right):N\ge2,\:1\le k<N\right\} \right).\]
 The claim then follows by showing that the latter two terms are equal
to $-\infty$. In fact, since $\exp\left(S_{\omega}\varphi\right)\le\exp\left(-\omega_{2}\right)$
we have, for every $\beta\in\R$, that \[
\sum_{n\ge2}\e^{S_{\left(n,n-1,\dots,1\right)}\left(\varphi-\beta\psi\right)}\le\sum_{n\ge2}\e^{-n+1}\left(n\left(n+1\right)\right)^{-\beta}<\infty\]
as well as \begin{eqnarray*}
\sum_{N\geq2}\sum_{1\le k<N}\!\!\!\!\e^{S_{\left(1,N,N-1,\dots,N-k+1\right)}\varphi-\beta\psi} & \le & \sum_{N\ge2}\e^{-N}\!\!\!\sum_{1\le k<N}\!\!\negthinspace\left(\frac{2N\left(N+1\right)}{\left(N-k+1\right)\left(N-k\right)}\right)^{-\beta}\\
 & < & \infty.\end{eqnarray*}
 
\end{example}

We finish this section by stating a variational principle generalising
results from \citep{MR2256622} to our situation. This result will
be crucial also for our results on the topological pressure for special
semi-flows in Section \ref{sec:Proof Special-semi-flows}.
\begin{prop}
[Variational Principle] \label{pro:variationalprinciple-1-1}Let
the incidence matrix of $\Sigma$ be irreducible and let $\varphi,\psi:\Sigma\rightarrow\R$
be Hölder continuous with $\psi>0$. Then we have for every $a\in I$\begin{eqnarray*}
\mathcal{P}_{\psi}\left(\varphi,\Sigma_{a}^{\mathrm{per}}\right) & = & \sup\left\{ \frac{h_{\nu}\left(\sigma\right)}{\int\psi d\nu}+\frac{\int\varphi d\nu}{\int\psi d\nu}:\nu\in\mathcal{M}_{\sigma}^{1}\:\mbox{with }\varphi,\psi\in L^{1}\left(\nu\right)\right\} \\
 & = & \sup\left\{ \frac{h_{\nu}\left(\sigma\right)}{\int\psi d\nu}+\frac{\int\varphi d\nu}{\int\psi d\nu}:\nu\in\mathcal{E}_{\sigma,c}^{1}\right\} ,\end{eqnarray*}
where $\mathcal{M}_{\sigma}^{1}$ denotes the set of $\sigma$-invariant
probability measures on $\Sigma$, and $\mathcal{E}_{\sigma,c}^{1}\subset\mathcal{M}_{\sigma}^{1}$
the subset of ergodic probability measures with compact support. \end{prop}
\begin{proof}
First we show \[
\mathcal{P}_{\psi}\left(\varphi,\Sigma_{a}^{\mathrm{per}}\right)\geq\sup\left\{ \frac{h_{\nu}\left(\sigma\right)}{\int\psi d\nu}+\frac{\int\varphi d\nu}{\int\psi d\nu}:\nu\in\mathcal{M}_{\sigma}^{1}\:\mbox{with }\varphi,\psi\in L^{1}\left(\nu\right)\right\} .\]
By Corollary \ref{cor:pressure_rescaling} we have $0\ge\mathcal{P}_{1}\left(\varphi-\beta\psi,\Sigma_{a}^{\mathrm{per}}\right)$
for $\beta>\mathcal{P}_{\psi}\left(\varphi,\Sigma_{a}^{\mathrm{per}}\right)$.
Using results from \citep[Theorem 3]{MR1738951} for the Gurevi\v c
pressure, which turn out to be valid also for irreducible incidence
matrices, we obtain \begin{eqnarray*}
0 & \ge & \mathcal{P}_{1}\left(\varphi-\beta\psi,\Sigma_{a}^{\mathrm{per}}\right)\\
 & \ge & \sup\left\{ h_{\nu}\left(\sigma\right)+\int\varphi d\nu-\beta\int\psi d\nu:\nu\in\mathcal{M}_{\sigma}^{1}\:\mbox{with }\varphi,\psi\in L^{1}\left(\nu\right)\right\} \\
 & = & \sup\left\{ \int\psi d\nu\left(\frac{h_{\nu}\left(\sigma\right)}{\int\psi d\nu}+\frac{\int\varphi d\nu}{\int\psi d\nu}-\beta\right):\nu\in\mathcal{M}_{\sigma}^{1}\:\mbox{with }\varphi,\psi\in L^{1}\left(\nu\right)\right\} ,\end{eqnarray*}
and hence, $\mathcal{P}_{\psi}\left(\varphi,\Sigma_{a}^{\mathrm{per}}\right)\ge\sup\left\{ \frac{h_{\nu}\left(\sigma\right)}{\int\psi d\nu}+\frac{\int\varphi d\nu}{\int\psi d\nu}:\nu\in\mathcal{M}_{\sigma}^{1}\:\mbox{with }\varphi,\psi\in L^{1}\left(\nu\right)\right\} $. 

Next, we are going to prove that $\mathcal{P}_{\psi}\left(\varphi,\Sigma_{a}^{\mathrm{per}}\right)\leq\sup\left\{ \frac{h_{\nu}\left(\sigma\right)}{\int\psi d\nu}+\frac{\int\varphi d\nu}{\int\psi d\nu}:\nu\in\mathcal{E}_{\sigma,c}^{1}\right\} $.
Since by Corollary \ref{cor:ExhaustGurevich}, $\left(\psi,\varphi,\Sigma_{a}^{\mathrm{per}}\right)$
satisfies the exhausting principle we may restrict our attention to
the finite alphabet case. More precisely, we find a sequence of compact
$\sigma$-invariant subsets $K_{n}\subset\Sigma$, such that $\lim_{n\rightarrow\infty}\mathcal{P}_{\psi,K_{n}}\left(\varphi,\Sigma_{a}^{\mathrm{per}}\right)=\mathcal{P}_{\psi}\left(\varphi,\Sigma_{a}^{\mathrm{per}}\right)$.
For the finite alphabet case we have by Remark \ref{rem:psipressure_finitealphabet}
and by the classical variational principle (cf.\ \citep{MR648108,MR0457675})
\begin{eqnarray*}
0 & = & \mathcal{P}_{1,K_{n}}\left(\varphi-\mathcal{P}_{\psi,K_{n}}\left(\varphi,\Sigma_{a}^{\mathrm{per}}\right)\psi,\Sigma_{a}^{\mathrm{per}}\right)\\
 & = & \sup\left\{ h_{\mu}+\int\varphi d\mu-\mathcal{P}_{\psi,K_{n}}\left(\varphi,\Sigma_{a}^{\mathrm{per}}\right)\int\psi d\mu:\mu\in\mathcal{E}_{\sigma,c}^{1},\mu\left(K_{n}\right)=1\right\} \\
 & = & \sup\left\{ \frac{h_{\mu}+\int\varphi d\mu}{\int\psi d\mu}:\mu\in\mathcal{E}_{\sigma,c}^{1},\mu\left(K_{n}\right)=1\right\} -\mathcal{P}_{\psi,K_{n}}\left(\varphi,\Sigma_{a}^{\mathrm{per}}\right).\end{eqnarray*}
The fact that $\sup\left\{ \frac{h_{\nu}\left(\sigma\right)}{\int\psi d\nu}+\frac{\int\varphi d\nu}{\int\psi d\nu}:\nu\in\mathcal{M}_{\sigma}^{1}\:\mbox{with }\varphi,\psi\in L^{1}\left(\nu\right)\right\} $
is greater or equal to $\sup\left\{ \frac{h_{\nu}\left(\sigma\right)}{\int\psi d\nu}+\frac{\int\varphi d\nu}{\int\psi d\nu}:\nu\in\mathcal{E}_{\sigma,c}^{1}\right\} $
then finishes the proof.
\end{proof}

\section{Proof of Theorem \ref{thm:taupressure_is_entropy} \label{sec:Proof Special-semi-flows}}

Let $\Phi=\left(\varphi_{t}\right)_{t\in\R_{>0}}$ be the special
semi-flow over \[
Y:=\left\{ \left(\omega,t\right)\in\Sigma\times\R^{+}:0\leq t\leq\tau(\omega)\right\} \slash\sim\]
with height function $\tau:\Sigma\rightarrow\R_{>0}$ defined by $\varphi_{t}(\omega,s)=(\omega,s+t)$. 

As a key lemma for the proof of Theorem \ref{thm:taupressure_is_entropy}
we show that $\mathbf{P}\left(g\mid\Phi\right)$ satisfies a certain
variational principle, where the supremum is taken over the set $\mathcal{E}_{\sigma}^{1}\left(\tau\right)$
of ergodic $\sigma$-invariant probability measures $\mu$ on the
base space satisfying $\int\tau d\mu<\infty$. Note, that this Lemma
generalises \citep[Lemma 1]{MR1627271} to our situation.
\begin{lem}
\label{lem:entropylemma}If $\tau:\Sigma\rightarrow\R_{>0}$ is Hölder
continuous satisfying $\sum_{i=0}^{\infty}\tau\circ\sigma^{i}=\infty$
and $g:Y\rightarrow\R$ has the property that $\Delta_{g}:\Sigma\rightarrow\R$
is Hölder continuous then we have\[
\mathbf{P}\left(g\mid\Phi\right)=\sup\left\{ \frac{h\left(\sigma,\mu\right)}{\int\tau d\mu}+\frac{\int\Delta_{g}d\mu}{\int\tau d\mu}:\mu\in\mathcal{E}_{\sigma}^{1}\left(\tau\right)\:\mbox{with }\Delta_{g}\in L^{1}\left(\mu\right)\right\} .\]
\end{lem}
\begin{proof}
This proof follows mainly the proof of the analogue statement in \citep{MR1627271}.
Let us denote the set of $\sigma$-invariant ergodic measures on $\Sigma$
satisfying $\int\tau\,\mathrm{d}\mu<\infty$ by $\mathcal{E}_{\sigma}\left(\tau\right)$.
Note, that there is a one-to-one correspondence between the $\Phi$-invariant
ergodic probability measures on $Y$ and the normalised product measures
$\overline{\mu\times\lambda}:=\left(\mu\times\lambda\right)\big|_{Y}/\int\tau\,\mathrm{d}\mu$
with $\mu\in\mathcal{E}_{\sigma}\left(\tau\right)$ and $\lambda$
denoting the Lebesgue measures on the fibers. By definition of \foreignlanguage{ngerman}{$\mathbf{P}\left(g\mid\Phi\right)$}
and the fact that $\int g\, d\left(\mu\times\lambda\right)\big|_{Y}=\int\Delta_{g}d\mu$
we have\[
\mathbf{P}\left(g\mid\Phi\right)=\sup\left\{ h\left(\varphi_{1},\overline{\mu\times\lambda}\right)+\frac{\int\Delta_{g}d\mu}{\int\tau\,\mathrm{d}\mu}:\mu\in\mathcal{E}_{\sigma}\left(\tau\right)\:\mbox{such that }\Delta_{g}\in L^{1}\left(\mu\right)\right\} .\]
 Since $\tau>0$ and $\sum_{i=0}^{\infty}\tau\circ\sigma^{i}=\infty$
we have by a result of Hopf \citep{0017.28301} that all elements
of $\mathcal{E}_{\sigma}\left(\tau\right)$ are conservative. If $\mu$
is $\sigma$-finite ergodic invariant and conservative (not necessarily
a probability) then $h\left(\sigma,\mu\right)$ will subsequently
be understood as in \citep{MR0218522}. We then notice that\foreignlanguage{ngerman}{
$h\left(\varphi_{1},\overline{\mu\times\lambda}\right)$} is equal
to $h\left(\sigma,\mu\right)/\int\tau d\mu$ for all $\mu\in\mathcal{E}_{\sigma}\left(\tau\right)$.
This follows e.g.\ from \citep[Theorem 1]{MR1627271} applied to
the natural extension of $\left(Y,\varphi_{1}\right)$ and using the
results on the relative metric entropy by Ledrappier and Walters \citep{MR0476995}.
We conclude that \[
\mathbf{P}\left(g\mid\Phi\right)=\sup\left\{ \frac{h\left(\sigma,\mu\right)}{\int\tau d\mu}+\frac{\int\Delta_{g}d\mu}{\int\tau\,\mathrm{d}\mu}:\mu\in\mathcal{E}_{\sigma}\left(\tau\right)\:\mbox{such that }\Delta_{g}\in L^{1}\left(\mu\right)\right\} .\]
 Let us now introduce the induced transformation on $E$ by setting
\[
\sigma_{E}^{*}\left(x\right):=\sigma^{N\left(x\right)}\left(x\right),x\in E,\]
with $N\left(x\right):=\inf\left\{ n\ge1:\sigma^{n}\left(x\right)\in E\right\} $
denoting the first return time to the set $E$. As shown in \citep[3.1]{MR0218522},
every Borel set $E$ with $0<\mu\left(E\right)<\infty$ is a sweep-out
set (i.e. $\mu\left(\Sigma\setminus\bigcup_{n\in\N}\sigma^{-n}\left(E\right)\right)=0$)
and hence we have $h\left(\sigma,\mu\right)=h\left(\sigma_{E}^{*},\mu_{|E}\right)$. 

Let us first assume that $\mathbf{P}\left(g\mid\Phi\right)<\infty$.
Then for $\epsilon>0$ we find $\mu\in\mathcal{E}_{\sigma}\left(\tau\right)$
 such that \[
\mathbf{P}\left(g\mid\Phi\right)-\epsilon\le\frac{h\left(\sigma,\mu\right)}{\int\tau d\mu}+\frac{\int\Delta_{g}d\mu}{\int\tau\,\mathrm{d}\mu}\leq\mathbf{P}\left(g\mid\Phi\right)<\infty.\]
Since $\int\tau d\mu<\infty$ and $\tau>0$ we find $a\in\Sigma^{*}$
such that $\left[a\right]$ hase finite and positive measure with
respect to $\mu$. We may normalise $\mu$ such that $\mu_{|\left[a\right]}=1$
without changing $\frac{h\left(\sigma,\mu\right)}{\int\tau d\mu}+\frac{\int\Delta_{g}d\mu}{\int\tau\,\mathrm{d}\mu}$.
In this situation we then have $h\left(\sigma,\mu\right)=h\left(\sigma_{\left[a\right]}^{*},\mu_{|\left[a\right]}\right)$. 

Our aim is to construct a sequence of finite invariant ergodic measures
$\left(\mu_{n,k}\right)_{n,k\in\N}$ on $\Sigma$, such that for $n,k$
large we have $h\left(\sigma,\mu_{n,k}\right)\ge h\left(\sigma,\mu\right)-\epsilon$
as well as \[
\left|\int\Delta_{g}d\mu_{n,k}-\int\Delta_{g}\,\mathrm{d}\mu\right|<\epsilon\;\mbox{ and }\:\left|\int\tau d\mu_{n,k}-\int\tau\,\mathrm{d}\mu\right|<\epsilon,\]
which then would finish the proof.

For this we define the countable set of words \[
C^{*}:=\left\{ \omega\in\Sigma^{*}:\omega a\in\Sigma^{*},\left[\omega a\right]\subset\left[a\right],N\big|_{\left[\omega a\right]}=\left|\omega\right|\right\} =\left\{ \gamma_{i}:i\in\N\right\} \]
and let $C_{n}^{*}:=\left\{ \gamma_{i}\in C^{*}:i\le n\right\} $
denote the set of its first $n\in\N$ elements. We have that $\left(\left[a\right],\sigma_{\left[a\right]}^{*}\right)$
is conjugated to the fullshift over the alphabet $C^{*}$, which we
denote by $\left(\Sigma_{C^{*}},\sigma^{*}\right)$. By the definition
of the measure theoretical entropy we have for every $k\in\N$ and
all $n\in\N$ sufficiently large \[
h\left(\sigma_{\left[a\right]}^{*},\mu_{|\left[a\right]}\right)\le-\frac{1}{k}\sum_{\omega\in\left(C^{*}\right)^{k}}\mu\left(\left[\omega\right]\right)\log\mu\left(\left[\omega\right]\right)\le-\frac{1}{k}\sum_{\omega\in\left(C_{n}^{*}\right)^{k}}\mu\left(\left[\omega\right]\right)\log\mu\left(\left[\omega\right]\right)+\epsilon.\]

Let $q_{n,k}:=\mu\left(\bigcup_{\omega\in\left(C_{n}^{*}\right)^{k}}\left[\omega\right]\right)$,
which is finite for fixed $k\geq\left|a\right|$ and every $n\in\mathbb{N}$,
and tends to $1$ for $n\to\infty$. For $k\geq\left|a\right|$ and
$n\in\mathbb{N}$ we define the Bernoulli measures $\mu_{n,k}^{*}$
on $\Sigma_{C^{*}}$ by requiring that $\mu_{n,k}^{*}\left(\left[\omega\right]\right):=\mu\left(\left[\omega\right]\right)/q_{n,k}$
for $\omega\in\left(C_{n}^{*}\right)^{k}$, and zero else. Its push
forward to $\limsup_{n}\sigma^{-n}\left(\left[a\right]\right)$ will
also be denoted by $\mu_{n,k}^{*}$. Hence, we obtain for the measure
theoretical entropy of the induced system \begin{eqnarray*}
h\left(\sigma_{\left[a\right]}^{*},\mu_{n,k}^{*}\right) & = & -\frac{1}{k}\sum_{\omega\in\left(C_{n}^{*}\right)^{k}}\mu_{n,k}^{*}\left(\left[\omega\right]\right)\log\mu_{n,k}^{*}\left(\left[\omega\right]\right)\\
 & = & -\frac{1}{k}\sum_{\omega\in\left(C_{n}^{*}\right)^{k}}\mu\left(\left[\omega\right]\right)\log\mu\left(\left[\omega\right]\right)-\log q_{n,k},\end{eqnarray*}
which tends to $-\frac{1}{k}\sum_{\omega\in\left(C^{*}\right)^{k}}\mu\left(\left[\omega\right]\right)\log\mu\left(\left[\omega\right]\right)$
for $n\rightarrow\infty$. Combining this with the above estimate
we have for all $k,n\in\N$ large enough that\[
h\left(\sigma_{\left[a\right]}^{*},\mu_{n,k}^{*}\right)\ge h\left(\sigma_{\left[a\right]}^{*},\mu_{|\left[a\right]}\right)-\epsilon.\]
We define for $n,k\in\N$ the $\sigma$-invariant and ergodic measure
$\mu_{n,k}$ on $\left(\Sigma,\mathcal{B}\right)$ by \[
\mu_{n,k}\left(E\right):=\int\sum_{i=0}^{N\left(x\right)-1}\1_{E}\circ\sigma^{i}\left(x\right)\,\mathrm{d}\mu_{n,k}^{*}\left(x\right),\: E\in\mathcal{B}\left(\Sigma\right),\]
which by Kac's formula \citep{MR0022323} has finite total mass \[
\mu_{n,k}\left(\Sigma\right)=\int N\,\mathrm{d}\mu_{n,k}^{*}=\sum_{\omega\in\left(C_{n}^{*}\right)^{k}}\left|\omega_{1}\right|\mu\left(\left[\omega\right]\right)<\infty.\]
The above estimates imply that for $k$ and $n$ sufficiently large
\[
h\left(\sigma,\mu_{n,k}\right)=h\left(\sigma_{\left[a\right]}^{*},\mu_{n,k}^{*}\right)\ge h\left(\sigma_{\left[a\right]}^{*},\mu_{|\left[a\right]}\right)-\epsilon=h\left(\sigma,\mu\right)-\epsilon.\]
We also introduce the induced potentials $\tau^{*},\Delta_{g}^{*}$
defined for $x\in\limsup_{n}\sigma^{-n}\left(\left[a\right]\right)$
by\[
\tau^{*}\left(x\right):=\sum_{i=0}^{N\left(x\right)-1}\tau\left(\sigma^{i}\left(x\right)\right)\;\mbox{ and }\:\Delta_{g}^{*}\left(x\right):=\sum_{i=0}^{N\left(x\right)-1}\Delta_{g}\left(\sigma^{i}\left(x\right)\right).\]
Recall that $\tau^{*}$ and $\Delta_{g}^{*}$ are defined $\mu$-a.e.
on $\left[a\right]$. Since $\tau$ and $\Delta_{g}$ are Hölder continuous
it follows by the remark at the beginning of Section \ref{sub:Construction-of-loop-spaces}
that $\tau^{*}$ and $\Delta_{g}^{*}$ are Hölder continuous as functions
on $\left(C^{*}\right)^{\N}$. Hence we can choose $k$ sufficiently
large, such that \[
\sup\left\{ \max\left\{ \left|\tau^{*}\left(x\right)-\tau^{*}\left(y\right)\right|,\left|\Delta_{g}^{*}\left(x\right)-\Delta_{g}^{*}\left(y\right)\right|\right\} :x,y\in\left[\omega\right],\omega\in\left(C^{*}\right)^{k}\right\} <\epsilon/2.\]
 Consequently, we have \[
\left|\int\tau^{*}\,\mathrm{d}\mu_{n,k}^{*}-\int\tau^{*}\1_{\bigcup_{\omega\in\left(C_{n}^{*}\right)^{k}}\left[\omega\right]}\,\mathrm{d}\mu_{|\left[a\right]}\right|<\epsilon/2\]
 and \[
\left|\int\Delta_{g}^{*}\,\mathrm{d}\mu_{n,k}^{*}-\int\Delta_{g}^{*}\1_{\bigcup_{\omega\in\left(C_{n}^{*}\right)^{k}}\left[\omega\right]}\,\mathrm{d}\mu_{|\left[a\right]}\right|<\epsilon/2.\]
By Kac's formula we have $\int\tau\,\mathrm{d}\mu=\int\tau^{*}\,\mathrm{d}\mu_{|\left[a\right]}$
and $\int\Delta_{g}\tau\,\mathrm{d}\mu=\int\Delta_{g}^{*}\,\mathrm{d}\mu_{|\left[a\right]}$.
The latter implies that $\tau^{*}$ and $\Delta_{g}^{*}$ are $\mu_{|\left[a\right]}$-integrable.
Since $q_{n,k}$ tends to $1$ for $n\rightarrow\infty$, we have
for large $n$ \[
\left|\int\tau^{*}\1_{\bigcup_{\omega\in\left(C_{n}^{*}\right)^{k}}\left[\omega\right]}\,\mathrm{d}\mu_{|\left[a\right]}-\int\tau^{*}\,\mathrm{d}\mu_{|\left[a\right]}\right|<\epsilon/2\]
and \[
\left|\int\Delta_{g}^{*}\1_{\bigcup_{\omega\in\left(C_{n}^{*}\right)^{k}}\left[\omega\right]}\,\mathrm{d}\mu_{|\left[a\right]}-\int\Delta_{g}^{*}\,\mathrm{d}\mu_{|\left[a\right]}\right|<\epsilon/2.\]
Again by Kac's formula we have$\int\tau\,\mathrm{d}\mu_{n,k}=\int\tau^{*}\,\mathrm{d}\mu_{n,k}^{*}$
and $\int\Delta_{g}\,\mathrm{d}\mu_{n,k}=\int\Delta_{g}^{*}\,\mathrm{d}\mu_{n,k}^{*}$.
Combining this with the above estimates we obtain $\left|\int\tau\,\mathrm{d}\mu_{n,k}-\int\tau\,\mathrm{d}\mu\right|<\epsilon$
as well as $\left|\int\Delta_{g}\,\mathrm{d}\mu_{n,k}-\int\Delta_{g}\,\mathrm{d}\mu\right|<\epsilon$. 

Now the proof for the case $\mathbf{P}\left(g\mid\Phi\right)=\infty$
follows along the same lines. 
\end{proof}
\begin{proof}
[Proof of Theorem  \ref{thm:taupressure_is_entropy}] By Lemma \ref{lem:entropylemma}
we have \[
\mathbf{P}\left(g\mid\Phi\right)=\sup\left\{ \frac{h\left(\sigma,\mu\right)}{\int\tau d\mu}+\frac{\int\Delta_{g}d\mu}{\int\tau d\mu}:\mu\in\mathcal{E}_{\sigma}^{1}\left(\tau\right)\:\mbox{with }\Delta_{g}\in L^{1}\left(\mu\right)\right\} .\]
The variational principle for the $\psi$--induced pressure obtained
in Proposition \ref{pro:variationalprinciple-1-1} applied to the
irreducible components of $\Sigma$ then shows that the latter is
equal to $\sup_{a\in I}\mathcal{P}_{\tau}\left(\Delta_{g},\Sigma_{a}^{\mathrm{per}}\right).$
The second equality follows from Corollary \ref{cor:urbgur-criterion}(2).\end{proof}
\begin{rem}
We would finally like to remark that if $\tau$ is bounded away from
zero, then there is a one-to-one correspondents between the $\Phi$--invariant
probability measures on $Y$ and the set $\mathcal{M}_{\sigma}^{1}$.
Using the first equality in Proposition \ref{pro:variationalprinciple-1-1}
it follows that the value of $\mathbf{P}\left(g\mid\Phi\right)$ stays
unchanged if we replace $\mathcal{E}_{\Phi}^{1}$ by the larger set
$\mathcal{M}_{\Phi}^{1}$.
\end{rem}
\providecommand{\bysame}{\leavevmode\hbox to3em{\hrulefill}\thinspace}
\providecommand{\MR}{\relax\ifhmode\unskip\space\fi MR }
\providecommand{\MRhref}[2]{%
  \href{http://www.ams.org/mathscinet-getitem?mr=#1}{#2}
}
\providecommand{\href}[2]{#2}


\begin{thebibliography}{GdlH01}

\bibitem[BI06]{MR2256622}
L.~Barreira and G.~Iommi, \emph{Suspension flows over countable {M}arkov
  shifts}, J. Stat. Phys. \textbf{124} (2006), no.~1, 207--230. \MR{MR2256622
  (2008f:37070)}

\bibitem[Bro85]{MR783536}
R.~Brooks, \emph{The bottom of the spectrum of a {R}iemannian covering}, J.
  Reine Angew. Math. \textbf{357} (1985), 101--114. \MR{783536 (86h:58138)}

\bibitem[Coh82]{MR678175}
J.~M. Cohen, \emph{Cogrowth and amenability of discrete groups}, J. Funct.
  Anal. \textbf{48} (1982), no.~3, 301--309. \MR{MR678175 (85e:43004)}

\bibitem[DGS76]{MR0457675}
M.~Denker, Chr. Grillenberger, and K.~Sigmund, \emph{Ergodic theory on compact
  spaces}, Lecture Notes in Mathematics, Vol. 527, Springer-Verlag, Berlin,
  1976. \MR{MR0457675 (56 \#15879)}

\bibitem[DL09a]{DastjerdiLamei}
D.A. Dastjerdi and S.~Lamei, \emph{Generating function for special flows over
  the 1-step countable topological {M}arkov chains}, preprint (2009).

\bibitem[DL09b]{DastjerdiLameib}
\bysame, \emph{Observing geometric codes from phase space}, preprint (2009).

\bibitem[GdlH01]{MR1929333}
R.~Grigorchuk and P.~de~la Harpe, \emph{Limit behaviour of exponential growth
  rates for finitely generated groups}, Essays on geometry and related topics,
  {V}ol. 1, 2, Monogr. Enseign. Math., vol.~38, Enseignement Math., Geneva,
  2001, pp.~351--370. \MR{MR1929333 (2003h:20078)}

\bibitem[GK01]{MR1901076}
B.~M. Gurevi{\v{c}} and S.~Katok, \emph{Arithmetic coding and entropy for the
  positive geodesic flow on the modular surface}, Mosc. Math. J. \textbf{1}
  (2001), no.~4, 569--582, 645, Dedicated to the memory of I. G. Petrovskii on
  the occasion of his 100th anniversary. \MR{MR1901076 (2003h:37040)}

\bibitem[Gri80]{MR599539}
R.~Grigorchuk, \emph{Symmetrical random walks on discrete groups},
  Multicomponent random systems, Adv. Probab. Related Topics, vol.~6, Dekker,
  New York, 1980, pp.~285--325. \MR{MR599539 (83k:60016)}

\bibitem[Gur69]{MR0263162}
B.~M. Gurevi{\v{c}}, \emph{Topological entropy of a countable {M}arkov chain},
  Dokl. Akad. Nauk SSSR \textbf{187} (1969), 715--718. \MR{MR0263162 (41
  \#7767)}

\bibitem[Gur70]{MR0268356}
\bysame, \emph{Shift entropy and {M}arkov measures in the space of paths of a
  countable graph}, Dokl. Akad. Nauk SSSR \textbf{192} (1970), 963--965.
  \MR{MR0268356 (42 \#3254)}

\bibitem[Hop37]{0017.28301}
E.~Hopf, \emph{Ergodentheorie}, {Ergebnisse der Mathematik und ihrer
  Grenz\-ge\-biete 5, H. 2}, J. Springer, Berlin, 1937.

\bibitem[HU99]{MR1732376}
P.~Hanus and M.~Urba{{\'n}}ski, \emph{A new class of positive recurrent
  functions}, Geometry and topology in dynamics ({W}inston-{S}alem, {NC},
  1998/{S}an {A}ntonio, {TX}, 1999), Contemp. Math., vol. 246, Amer. Math.
  Soc., Providence, RI, 1999, pp.~123--135. \MR{MR1732376 (2000j:37006)}

\bibitem[Iom10]{Iommi2010}
G.~Iommi, \emph{Thermodynamic formalism for the positive geodesic flow on the
  modular surface}, preprint (2010).

\bibitem[JK11]{JaerischKessebohmer:09}
J.~Jaerisch and M.~Kesseb{\"o}hmer, \emph{Regularity of multifractal spectra of
  conformal iterated function systems}, Trans. Amer. Math. Soc. \textbf{363}
  (2011), no.~1, 313--330.

\bibitem[Kac47]{MR0022323}
M.~Kac, \emph{On the notion of recurrence in discrete stochastic processes},
  Bull. Amer. Math. Soc. \textbf{53} (1947), 1002--1010. \MR{0022323 (9,194a)}

\bibitem[Kes59a]{MR0112053}
H.~Kesten, \emph{Full {B}anach mean values on countable groups}, Math. Scand.
  \textbf{7} (1959), 146--156. \MR{MR0112053 (22 \#2911)}

\bibitem[Kes59b]{MR0109367}
\bysame, \emph{Symmetric random walks on groups}, Trans. Amer. Math. Soc.
  \textbf{92} (1959), 336--354. \MR{MR0109367 (22 \#253)}

\bibitem[Kes67]{MR0214137}
\bysame, \emph{The {M}artin boundary of recurrent random walks on countable
  groups}, Proc. {F}ifth {B}erkeley {S}ympos. {M}ath. {S}tatist. and
  {P}robability ({B}erkeley, {C}alif., 1965/66), Univ. California Press,
  Berkeley, Calif., 1967, pp.~Vol. II: Contributions to Probability Theory,
  Part 2, pp. 51--74. \MR{MR0214137 (35 \#4988)}

\bibitem[Kes01]{MR1819804}
M.~Kesseb{{\"o}}hmer, \emph{Large deviation for weak {G}ibbs measures and
  multifractal spectra}, Nonlinearity \textbf{14} (2001), no.~2, 395--409.
  \MR{MR1819804 (2002a:60037)}

\bibitem[KMS10]{MundayKessStrat10}
M.~Kesseb\"ohmer, S.~Munday, and B.~O. Stratmann, \emph{Strong renewal theorems
  and {L}yapunov spectra for $\alpha$--{F}arey and $\alpha$--{L}\"uroth
  systems}, arXiv (2010).

\bibitem[Kre67]{MR0218522}
U.~Krengel, \emph{Entropy of conservative transformations}, Z.
  Wahrscheinlichkeitstheorie und Verw. Gebiete \textbf{7} (1967), 161--181.
  \MR{0218522 (36 \#1608)}

\bibitem[KS07]{MR2337557}
M.~Kesseb{{\"o}}hmer and B.~O. Stratmann, \emph{Homology at infinity; fractal
  geometry of limiting symbols for modular subgroups}, Topology \textbf{46}
  (2007), no.~5, 469--491. \MR{MR2337557 (2009a:37096)}

\bibitem[LW77]{MR0476995}
F.~Ledrappier and P.~Walters, \emph{A relativised variational principle for
  continuous transformations}, J. London Math. Soc. (2) \textbf{16} (1977),
  no.~3, 568--576. \MR{0476995 (57 \#16540)}

\bibitem[MU01]{MR1853808}
R.~D. Mauldin and M.~Urba{{\'n}}ski, \emph{Gibbs states on the symbolic space
  over an infinite alphabet}, Israel J. Math. \textbf{125} (2001), 93--130.
  \MR{MR1853808 (2002k:37048)}

\bibitem[MU03]{MR2003772}
\bysame, \emph{Graph directed {M}arkov systems}, Cambridge Tracts in
  Mathematics, vol. 148, Cambridge University Press, Cambridge, 2003, Geometry
  and dynamics of limit sets. \MR{MR2003772 (2006e:37036)}

\bibitem[Prz99]{MR1615954}
F.~Przytycki, \emph{Conical limit set and {P}oincar{\'e} exponent for
  iterations of rational functions}, Trans. Amer. Math. Soc. \textbf{351}
  (1999), no.~5, 2081--2099. \MR{MR1615954 (99h:58110)}

\bibitem[RSU08]{MR2439479}
M.~Roy, H.~Sumi, and M.~Urba{{\'n}}ski, \emph{Analytic families of holomorphic
  iterated function systems}, Nonlinearity \textbf{21} (2008), no.~10,
  2255--2279. \MR{MR2439479 (2009h:37096)}

\bibitem[Sar99]{MR1738951}
O.~M. Sarig, \emph{Thermodynamic formalism for countable {M}arkov shifts},
  Ergodic Theory Dynam. Systems \textbf{19} (1999), no.~6, 1565--1593.
  \MR{MR1738951 (2000m:37009)}

\bibitem[Sar01]{MR1818392}
\bysame, \emph{Thermodynamic formalism for null recurrent potentials}, Israel
  J. Math. \textbf{121} (2001), 285--311. \MR{MR1818392 (2001m:37059)}

\bibitem[Sar03]{MR1955261}
\bysame, \emph{Existence of {G}ibbs measures for countable {M}arkov shifts},
  Proc. Amer. Math. Soc. \textbf{131} (2003), no.~6, 1751--1758 (electronic).
  \MR{MR1955261 (2004b:37056)}

\bibitem[Sav98]{MR1627271}
S.~V. Savchenko, \emph{Special flows constructed from countable topological
  {M}arkov chains}, Funktsional. Anal. i Prilozhen. \textbf{32} (1998), no.~1,
  40--53, 96. \MR{MR1627271 (99m:28040)}

\bibitem[Wal82]{MR648108}
P.~Walters, \emph{An introduction to ergodic theory}, Graduate Texts in
  Mathematics, vol.~79, Springer-Verlag, New York, 1982. \MR{MR648108
  (84e:28017)}

\end{thebibliography}
\end{document}